\documentclass[11pt,fleqn,reqno]{article}
\usepackage{epsfig,amssymb,latexsym,amsmath,pifont,multicol,mathtools,amsmath,verbatim,amsthm,float,lipsum}
\usepackage{caption} 
\usepackage[utf8]{inputenc}
\usepackage[T1]{fontenc}

\usepackage[varg]{txfonts}
\let\mathbb=\varmathbb

\usepackage[sans]{dsfont}

\usepackage[left=2.5cm, right=2.5cm, top=2.5cm, bottom=2.5cm]{geometry}

\usepackage[%
cal=euler,
bb=fourier,
scr=zapfc,
]
{mathalfa}

\usepackage[font=small,labelfont=bf]{caption}
\usepackage{subfigure}
\usepackage{graphicx}
\graphicspath{{Figures/}} 
\usepackage{acronym}
\usepackage{latexsym}
\usepackage{paralist}
\usepackage{wasysym}
\usepackage{xspace}
\usepackage{framed}
\usepackage{palatino,pxfonts}
\usepackage{authblk}
\usepackage{enumitem}
\usepackage[numbers,sort&compress]{natbib}

\usepackage[dvipsnames,svgnames]{xcolor}
\colorlet{MyBlue}{DodgerBlue!75!Black}
\colorlet{MyGreen}{DarkGreen!95!Black}

\usepackage{hyperref}
\hypersetup{
colorlinks=true,
linktocpage=true,
pdfstartview=FitH,
breaklinks=true,
pdfpagemode=UseNone,
pageanchor=true,
pdfpagemode=UseOutlines,
plainpages=false,
bookmarksnumbered,
bookmarksopen=false,
bookmarksopenlevel=1,
hypertexnames=true,
pdfhighlight=/O,
urlcolor=MyBlue!60!black,linkcolor=MyBlue!70!black,citecolor=DarkGreen!70!black, 
pdftitle={},
pdfauthor={},
pdfsubject={},
pdfkeywords={},
pdfcreator={pdfLaTeX},
pdfproducer={LaTeX with hyperref}
}

\numberwithin{equation}{section}  
\usepackage[sort&compress,capitalize,nameinlink]{cleveref}
\crefname{example}{Ex.}{Exs.}

\crefrangeformat{equation}{\upshape(#3#1#4)\textendash(#5#2#6)}

\usepackage[textsize=footnotesize]{todonotes}
\definecolor{ms}{rgb}{0,0,0.7}

 \definecolor{tvl}{rgb}{0.7,0.1,0.1}

\definecolor{sw}{rgb}{0.1,0.5,0.1}

\usepackage{mathtools}
\usepackage[para,online,flushleft]{threeparttable}
\usepackage{multirow}

\usepackage{physics}
\usepackage{float}
\newcommand{\eps}{\varepsilon}
\DeclareMathOperator*{\argmin}{argmin}

\DeclareMathOperator{\zer}{zer}
\DeclareMathOperator{\conv}{Conv}

\DeclareMathOperator{\dist}{dist}
\DeclareMathOperator{\dif}{d\!}

\DeclareMathOperator{\dom}{dom}

\DeclareMathOperator{\Int}{int}

\DeclareMathOperator{\Id}{Id}

\newcommand{\ca}{\mathtt{a}}
\newcommand{\cb}{\mathtt{b}}

\newcommand{\cs}{\mathtt{s}}

\DeclareMathOperator{\prox}{\mathsf{prox}}

\newcommand{\bL}{\mathbf{L}}

\newcommand{\bK}{\mathbf{K}}

\newcommand{\B}{\mathbb{B}}

\newcommand{\bx}{x}
\newcommand{\by}{y}
\newcommand{\bu}{u}

\newcommand{\bz}{z}

\newcommand{\bv}{\mathbf{v}}

\renewcommand{\iff}{\Leftrightarrow}

\newcommand{\eqdef}{\triangleq}

\newcommand{\scrA}{\mathcal{A}}
\newcommand{\scrB}{\mathcal{B}}

\newcommand{\scrD}{\mathcal{D}}
\newcommand{\scrE}{\mathcal{E}}
\newcommand{\scrF}{\mathcal{F}}
\newcommand{\scrG}{\mathcal{G}}

\newcommand{\scrL}{\mathcal{L}}

\newcommand{\scrO}{\mathcal{O}}
\newcommand{\scrP}{\mathcal{P}}

\newcommand{\scrW}{\mathcal{W}}
\newcommand{\scrX}{\mathcal{X}}
\newcommand{\scrY}{\mathcal{Y}}

\renewcommand{\Pr}{\mathbb{P}}
\newcommand{\Ex}{\mathbb{E}}

\newcommand{\measP}{{\mathsf{P}}}

\newcommand{\Leb}{{\mathsf{Leb}}}

\newcommand{\Normal}{\mathsf{N}}
\newcommand{\Uniform}{\mathsf{U}}

\newcommand{\0}{\mathbf{0}}

\newcommand{\Rn}{\R^n}

\newcommand{\R}{\mathbb{R}}

\newcommand{\N}{\mathbb{N}}

\newcommand{\bC}{{\mathbf{C}}}

\DeclareMathOperator{\Opt}{Opt}							

\newcommand{\lip}{\mathsf{lip}}

\usepackage{algorithmicx}
\usepackage{algpseudocode}								
\usepackage{algorithm}

\theoremstyle{plain}
\newtheorem{theorem}{Theorem}
\newtheorem{corollary}[theorem]{Corollary}
\newtheorem*{corollary*}{Corollary}
\newtheorem{lemma}[theorem]{Lemma}
\newtheorem{proposition}[theorem]{Proposition}

\newtheorem*{hypothesis}{Standing Hypothesis}
\theoremstyle{definition}
\newtheorem{definition}[theorem]{Definition}
\newtheorem*{definition*}{Definition}
\newtheorem{assumption}{Assumption}

\newcommand{\close}{\hfill{\footnotesize$\Diamond$}}
\theoremstyle{remark}
\newtheorem{remark}{Remark}
\newtheorem*{remark*}{Remark}
\newtheorem*{notation*}{Notational remark}
\newtheorem{example}{Example}

\numberwithin{theorem}{section}
\numberwithin{remark}{section}
\numberwithin{example}{section}

\DeclarePairedDelimiter{\inner}{\langle}{\rangle}
\DeclarePairedDelimiter{\dnorm}{\lVert}{\rVert_{\ast}}

\title{Derivative-free stochastic bilevel optimization for inverse problems}
\date{\today}

\author[1]{\small Mathias Staudigl, Simon Weissmann}
\author[2,3]{\small Tristan van Leeuwen}

\affil[1]{\footnotesize Mannheim University, Department of Mathematics, B6 26, 68159 Mannheim\\
(\href{mailto:mathias.staudigl@uni-mannheim.de}{mathias.staudigl@uni-mannheim}, \href{mailto:simon.weißmann@uni-mannheim.de}{simon.weissmann@uni-mannheim.de})
}

\affil[2]{\footnotesize Centrum Wiskunde \& Informatica, Science Park Amsterdam 123, 1098 XG Amsterdam, The Netherlands \\
(\href{mailto:T.van.Leeuwen@cwi.nl}{T.van.Leeuwen@cwi.nl})}

\affil[3]{\footnotesize Utrecht University, Heidelberglaan 8, 3584 CS Utrecht, The Netherlands}

\begin{document}

\maketitle

\begin{abstract}%
%

Inverse problems are key issues in several scientific areas, including signal processing and medical imaging. Data-driven approaches for inverse problems aim for learning model and regularization parameters from observed data samples, and investigate their generalization properties when confronted with unseen data. This approach dictates a statistical approach to inverse problems, calling for stochastic optimization methods. In order to learn model and regularisation parameters simultaneously, we develop in this paper a stochastic bilevel optimization approach in which the lower level problem represents a variational reconstruction method formulated as a convex non-smooth optimization problem, depending on the observed sample. The upper level problem represents the learning task of the regularisation parameters. Combining the lower level and the upper level problem leads to a stochastic non-smooth and non-convex optimization problem, for which standard gradient-based methods are not straightforward to implement. Instead, we develop a unified and flexible methodology, building on a derivative-free approach, which allows us to solve the bilevel optimization problem only with samples of the objective function values. We perform a complete complexity analysis of this scheme. Numerical results on signal denoising and experimental design demonstrate the computational efficiency and the generalization properties of our method. 
\end{abstract}
\section{Introduction}
\label{sec:Intro}
%
Bilevel optimization is a very important optimization methodology for solving inverse problems \cite{EHN1996,benning_burger_2018}. The strength of bilevel optimization is that it allows to endogenously learn hyper-parameters, which otherwise would have to be tuned manually. A very prominent instantiation of this is the task of learning regularization parameters \cite{haber2003learning,Kunisch:2013aa,Holler:2018aa}. A mathematical formulation of this problem is to first define a variational reconstruction method involving a data fidelity function $x\mapsto \scrL(\bK(x),\xi)$, where $\xi\in\Xi$ is the observed image, and $\bK:\scrX\to\scrD$ is the forward operator, mapping model parameters $x$ to observations in $\scrD$. We then define the reconstruction operator $x(y,\cdot):\Xi\to\scrX$ as a solution to the optimization problem
\begin{equation}\label{eq:VarHyperparameter}
x(y,\xi)\in\argmin_{x\in\scrX}\{\scrL(\bK(x),\xi)+R_{y}(x)\}\qquad \text{for }(y,\xi)\in\scrY\times \Xi.
\end{equation}
The function $R_{y}:\scrX\to\R\cup\{+\infty\}$ is a parameter-depending regularizer, that avoids overfitting and imposes a-priori known structure into the model parameter. Choosing this parameter $y\in\scrY$ a-priori is a severe bottleneck in the effective solution of the underlying inverse problem and poses significant practical challenges. Traditionally, this problem of \emph{hyperparameter tuning} has been heuristically solved and generally requires a large number of solutions of this variational problem for a pre-defined grid of parameter values $y$. Bilevel optimization replaces this heuristic search procedure by a disciplined optimization approach which selects model parameters on par with regularization parameters, given the data sample representing the inverse problem. However, the bilevel methodology is not only useful for solving the hyperparameter learning problem. It also has a significant impact for other inverse problems in which the forward operator itself exhibits a dependence on model parameters. This is generically the case in \emph{optimal experimental design}. In this framework we address the question where and when to take measurements, which variables to include, and what experimental conditions should be employed. Mathematically, this leads to a forward model $\bK_{y}$ which depends on a vector of design parameters $y\in\scrY$, which have to be chosen before the variational model is solved. Hence, problem \eqref{eq:VarHyperparameter} needs to be modified to
\begin{equation}\label{eq:OE}
x(y,\xi)\in\argmin_{x\in\scrX}\{\scrL(\bK_{y}(x),\xi)+R_{y}(x)\}\qquad \text{for }(y,\xi)\in\scrY\times\Xi.
\end{equation}
To obtain a generic set-up for learning selected components of \eqref{eq:OE} from data we adopt a supervised learning approach \cite{DataInverse}: We are given random variables $\xi=(\xi_{1},\xi_{2})\in\Xi_{1}\times\Xi_{2}=\Xi$, defined on a fixed probability space $(\Omega,\scrF,\Pr)$, where the first component contain model parameters, and the second component are the observations. This random element lives in some measurable space $\Xi$ with joint distribution $\measP_{\xi}$. Our aim is to learn the model parameters $x^{\ast}(y,\xi_{2})$ (as a function of regularization parameters and data), and regularization parameters $y^{\ast}\in\scrY$ simultaneously so that they are optimal given the expected risk defined in terms of the loss function and the data. This leads to the stochastic bilevel formulation
\begin{equation}\label{eq:SBL}
\begin{split}
y^{\ast}\in\argmin_{y\in\scrY}\{\Ex_{\xi}[f(x^{\ast}(y,\xi_{2}),\xi_{1})]+r_{1}(y)\}\\ 
x^{\ast}(y,\xi_{2})\in\argmin_{x\in\scrX}\{g(x,y,\xi_{2})+r_{2}(x)\}
\end{split}
\end{equation}
The \emph{upper level objective}  $\Ex_{\xi}[f(x^{\ast}(y,\xi_{2}),\xi_{1})]+r_{1}(y)$ contains an expectation-valued part involving a tracking-type function $f:\scrX\times\Xi_{1}\to\R$, usually assumed to be sufficiently smooth, and a regularizer/penalty function $r_{2}(y)$, i.e.~chosen to promote sparsity in the parameter vector. The \emph{lower level objective} $g(x,y,\xi_{2})+r_{2}(x)$ is a variational model for obtaining model parameters, as a function of the realized data $\xi_{2}\in\Xi_{2}$ and the tuneable hyperparameter $y\in\scrY$.

\begin{example}[Bilevel Learning]
The bilevel learning approach for inverse problem is a statistical learning methodology to select the regularization parameter in minimization based inverse problems. The unknown parameter and the corresponding observation are modelled as jointly distributed random variables $(X,D):\Omega\to \Xi_1\times  \Xi_2=\scrX\times\scrD$, defined as
\[D(\omega) = \bK_{y_1} X(\omega) + Z(\omega),\quad \omega\in\Omega\,, \]
where $Z:\Omega\to\Xi_2$ denotes measurements noise. The forward operator may depend explicitly on a hyperparameter $y_1\in\scrY_1$. In order to build a reconstruction of the unknown parameter $X(\omega)$ for a fixed $\omega\in\Omega$, we consider the variational form of the inverse problem 
\[\min_{x\in\scrX}\ g(x,(y_1,y_2),D(\omega)),\quad g(x,(y_1,y_2),D(\omega))\eqdef \scrL(\bK_{y_1}(x),D(\omega))+R_{y_2}(x)\,, \]
where $\scrL:\scrD\times \scrD\to\R$ is a data fidelity function and $R_{y_2}:\scrX\to\R$ is a regularization function with regularization parameter $y_2\in\scrY$. The reconstruction highly depends on the choice of hyperparameter $(y_1,y_2)\in\scrY\eqdef \scrY_1\times \scrY_2$ and the goal in bilevel learning is to choose these hyperparameters based on the stochastic bilevel optimization problem \eqref{eq:SBL}, where $(X,D)$ take the role of $(\xi_1,\xi_2)$. This approach has been investigated in many previous studies (see e.g. \cite{Ochs:2016aa,Holler:2018aa,Ehrhardt:2021ab}). We will provide more details about this application in Section~\ref{sec:numerics}.
\close
\end{example}

\subsection{Challenges and related literature}
Directly solving the stochastic bilevel optimization problem \eqref{eq:SBL} is challenging for at least two reasons: First, in order to solve the upper level problem, we need to know a solution of the lower level problem. However, this is just our variational inverse problem, and thus is typically a large-scale optimization problem itself (although very often convex). Even if this can entail computational challenges, it can in principle be overcome via state-of-the art convex programming techniques; The second challenge that arises is how to optimize the upper level objective function, which is only available as an implicit function of the lower level solution mapping $x^{\ast}(y,\xi)$. This problem becomes even more pronounced when the lower level solution is not unique. While non-uniqueness could be dealt with penalty methods (see e.g. \cite{Kwon:2023aa,FOMPenalty}), the presence of stochastic perturbations in the problem data, renders also this approach challenging. Instead, in this paper we investigate in detail solution methods for settings in which the lower level mapping can be solved up to some accuracy at reasonable computational costs, and then use this mapping to construct a simple optimization method that avoids delicate issues such as computing gradients, or even higher-order information of the upper level objective. Specifically, we make the following standing hypothesis throughout this paper\footnote{A more precise formulation of this hypothesis will be given in Section \ref{sec:Problem}.}: 
\begin{hypothesis}
    The lower level solution problem admits a unique solution $x^{\ast}(y,\xi_{2})$, which is a measurable function of the data $\xi_{2}$. 
\end{hypothesis}
Working under this hypothesis, the main remaining question is how to effectively solve the upper level problem 
\begin{equation}\label{eq:Psi}
 \min_{y\in\scrY}   \Psi(y)\eqdef \Ex_{\xi}[F(x^{\ast}(y,\xi_{2}),\xi_{1})]+r_{1}(y).
\end{equation}
The challenge within this formulation lies in the fact that the first function $y\mapsto \Ex_{\xi}[F(x^{\ast}(y,\xi_{2}),\xi_{1})]$ is expectation-valued (hence hard to evaluate) and in general non-smooth and non-convex. The lack of regularity properties make a direct gradient-based approach less qualified, without even talking about the difficulties in computing a gradient (aka the \emph{hypergradient} \cite{Grazzi:2020aa,Ehrhardt:2024aa}) of this composite function. The key complications arising in this formulation are (i) the dependence of the lower level solution $\bx^{\ast}(\by,\xi_{2})$ on the random variable $\xi_{2}$, (ii) the potential non-smoothness of the lower level variational problem, (iii) the non-smoothness of the upper level problem. All three complications make any attempt to adapt standard methods for solving bilevel optimization problems complicated. One main contribution of this paper is to construct a practically efficient strategy for solving the stochastic bilevel problem \eqref{eq:SBL} building on a zeroth-order stochastic oracle model for estimating the hypergradient, allowing for bias in the random estimator, and inexactness of the solution of the lower level problem. Although this setting received a significant amount of attention recently, mainly driven from applications in machine learning such as meta-learning \cite{rajeswaran2019meta}, hyper-parameter optimization \cite{franceschi2018bilevel,sinha2024gradient} and reinforcement learning \cite{hong2023two}, the composite setting embodied in \eqref{eq:Psi} is complicating the hypergradient estimation task a lot. The survey \cite{BilevelLearningIEEE} gives a comprehensive state-of-the-art overview. A technical contribution of this paper is to construct a practically efficient strategy for solving the stochastic bilevel problem \eqref{eq:SBL}, building on a \emph{zeroth-order stochastic oracle} model for estimating the hypergradient, allowing for bias in the random estimator, and inexactness of the solution of the lower level problem. 

\paragraph{Stochastic Bilevel Optimization}
The bilevel instance \eqref{eq:SBL} differs from the typical machine learning setting in our requirement that the lower level problem needs to be solved for any realization of the random variable $\xi_{2}$. In machine learning, the typically encountered formulation has no non-smooth terms and no explicit constraints:
\[
\min_{\by\in \R^{d}} \psi(\by)=f(\bx^{\ast}(\by),\by)\quad\text{s.t.: } \bx^{\ast}(\by)\in\argmin_{x\in\Rn}g(\bx,\by)
\]
where $f(x,y)=\Ex[F(\bx,\by,\xi_{1})]$ and $g(\bx,\by)=\Ex[G(\bx,\by,\xi_{2})]$. Under strong regularity conditions the hyperobjective $\psi$ is smooth enough so that its gradient can be characterized by the implicit function theorem 
$$
\nabla\psi(\by)=\nabla_{y} f(\bx^{\ast}(\by),\by)-\nabla^{2}_{xy}g(\bx^{\ast}(\by),\by)\left[\nabla^{2}_{xx}g(\bx^{\ast}(\by),\by)\right]^{-1}\nabla_{x}f(\bx^{\ast}(\by),\by).
$$
In the composite non-smooth setting arising in inverse problems, and which is of interest in this paper, there is no hope that a similar formula for the hypergradient can be defined. Recently, \cite{Cui:2022aa} propose a stochastic zeroth-order method for a class of stochastic mathematical programs under equilibrium constraints, in which the lower-level problem is described by the solution set of a stochastic variational inequality, and the upper-level problem is a stochastic unconstrained optimization problem. We extend this setting to the non-smooth proximal framework in both the upper and the lower-level problem. This is a non-trivial extension, since it requires a fundamentally different analysis of the iteration complexity of the method in terms of the prox-gradient mapping (cf. \eqref{eq:gradient}). Moreover, we provide complexity estimates on the criticality measure represented by the prox-gradient mapping via an integrated smoothing and zeroth-order optimization scheme, without any a-priori convexity assumptions on the hyperobjective.

\paragraph{Zeroth-order stochastic optimization}
The numerical solution of stochastic optimization problems requires the availability of a stochastic oracle. In low informational settings such as simulation-based, or black-box optimization, an attractive stochastic oracle is one that relies only on noisy function queries. Such zeroth-order methods have been studied in the literature under the name of derivative-free optimization \cite{conn2009introduction,spall2005introduction}, Bayesian optimization \cite{garnett_bayesoptbook_2023}, and optimization with bandit feedback \cite{CesLug06,Duvocelle:2022aa}. Moreover, gradient-free methods received a lot of attention within mathematical imaging \cite{Ehrhardt:2024aa,Ehrhardt:2021ab}, and scientific computing \cite{Kozak:2022aa,Pougkakiotis:2023aa}, as well as in machine learning and computational statistics \cite{agarwal2010optimal,Duchi:2015aa,Ghadimi:2013aa}. We discuss the connection to the most important references in the following. 

\cite{Berahas:2022aa} performs a detailed comparison of different derivative-free methods based on noisy function evaluations, assuming that the noise component is additive and with zero mean and bounded range. They established conditions on the gradient estimation errors that guarantee convergence to a neighborhood of the solution. We perform a complexity analysis of a derivative free method in which the function values are noisy evaluations of the hyperobjective of the bilevel problem \eqref{eq:SBL}, without a uniformly bounded noise assumption. Instead, we only assume standard variance bounds in $L^{p}$, for some $p\geq 2$. 

\cite{Balasubramanian:2022aa} provide an in-depth analysis of zero-order estimators for solving general stochastic optimization problems, using a Frank-Wolfe method, a stochastic proximal gradient method, or a higher-order method building on the cubic regularization globalization technique. Their general complexity statements are not immediately transferable to our problem, since we solve a stochastic bilevel problem, with potentially inexact feedback between the upper and the lower level problem. This noisy and inexact feedback mechanism leads to an additional bias in the gradient estimator, which needs to be carefully balanced in order to prove convergence guarantees of the method.

\subsection{Main Contributions and outline}
Our main results can be summarized as follows:
\begin{enumerate}
    \item Under weak regularity assumption on the hyperobjective $h(y)=\Ex[F(x^{\ast}(y,\xi_{2}),\xi_{1})]$ (essentially only Lipschitz continuity), we derive an iteration complexity statement in terms of the proximal gradient mapping for the Gaussian smoothed objective $h_{\eta}$. In particular, we give complexity statements assuming that the lower level problem can be solved exactly, or inexactly, with a controlled precision in an $L^{p}$ sense. 
   \item We particularize this result in the convex case to obtain a complexity statement in terms of the original objective function optimality gap.  
   \item To relate the complexity statement derived for the smoothed hyperobjective, we define a notion for a relaxed stationary point, using a fuzzy version of the Goldstein subgradient, originally introduced in \cite{Gold77} for Lipschitz continuous mathematical programs. This allows us to transfer the complexity statements derived in pervious sections for the smoothed prox-gradient mapping to a criticality measure involving the Goldstein subgradient. 
\end{enumerate}

The remainder of the manuscript is structured as follows. We introduce our notation and some known results, used in the analysis, in Section~\ref{sec:prelims}. Section~\ref{sec:Problem} presents the formulation of the stochastic bilevel optimization problem with the corresponding assumptions. In Section~\ref{sec:ZeroOrder}, we introduce our proposed zeroth-order optimization method. Section~\ref{sec:Nonconvex} begins the convergence analysis in a non-convex setting with a fixed smoothing parameter, covering both exact and inexact lower level solutions. We then proceed to Section~\ref{sec:convex}, where we analyze the convex case and quantify the smoothing error. Section~\ref{sec:Goldstein} addresses the explicit complexity and relaxed stationarity for non-convex problems. In Section~\ref{sec:numerics}, we apply our algorithm to linear inverse problems, with a particular focus on imaging. Finally, we conclude the main body of the manuscript with a summary in Section~\ref{sec:conclusion}. For clarity, most of the proofs are deferred to Appendices~\ref{app:Smoothing}--\ref{app:2}. 
\section{Notation and Preliminaries}
\label{sec:prelims}
%
For a finite dimensional real vector space $\scrE$, we denote by $\scrE^{\ast}$ its dual space. The value of a linear function $s\in\scrE^{\ast}$ at point $x\in\scrE$ is denoted by $s(x)\eqdef\inner{s,\bx}$. We endow the spaces $\scrE$ and $\scrE^{\ast}$ with Euclidean norms 
$\norm{\bx}=\inner{B\bx,\bx}^{1/2}$ and $\norm{s}_{\ast}=\inner{s,B^{-1}s}^{1/2}$, where $B=B^{\ast}$ represents the Riesz isomorphism, i.e. a positive definite linear operator from $\scrE$ to $\scrE^{\ast}$. For a subset $C\subset\scrE$ we define the distance of $x\in\scrE$ to $M$ by $\dist(x,C)\eqdef \inf_{z\in C}\norm{x-z}$. The ball with center $x$ and radius $r>0$ is denoted as $\B(x,r)$. The convex hull of a set $X$ is denoted as $\conv(X)$.  If $\Omega$ is a topological space, we denote by $\scrB(\Omega)$ the Borel $\sigma$-algebra.  In this paper we consider functions with different levels of smoothness. We say a function $h:\scrE\to\R$ belongs to class $\bC^{0,0}(\scrE)$ if there exists a constant $\lip_{0}(h)>0$ such that 
\[
|{h(\bx_{1})-h(\bx_{2})}|\leq \lip_{0}(h)\norm{\bx_{1}-\bx_{2}}, \qquad\forall \bx_{1},\bx_{2}\in\scrE\,,
\]
$h$ belongs to class $\bC^{1,1}(\scrE)$ if 
\[
\norm{\nabla h(\bx_{1})-\nabla h(\bx_{2})}_{\ast}\leq \lip_{1}(h)\norm{\bx_{1}-\bx_{2}},\qquad \forall x_1,x_2\in\scrE.
\]
For $h\in\bC^{1,1}(\scrE)$, we have the Lipschitz descent Lemma \cite[Lemma 1.2.3]{Nes18}
\begin{equation}\label{eq:descent}
h(\bx_{2})\leq h(\bx_{1})+\inner{\nabla h(\bx_{1}),\bx_{2}-\bx_{1}}+\frac{\lip_{1}(h)}{2}\norm{\bx_{2}-\bx_{1}}^2,\quad \forall x_1,x_2\in\scrE.
\end{equation}

For extended real-valued convex functions $h:\scrE\to[-\infty,\infty]$, we define its (effective) domain $\dom(h)=\{y\in\scrY\vert h(y)<\infty\}$. The convex subdifferential is the set-valued mapping $\partial h(\by)\eqdef \{v\in\scrE^{\ast}\vert h(\tilde{\by})\geq h(\by)+\inner{v,\tilde{\by}-\by}\quad\forall \tilde{\by}\in\scrE\}$. Elements of this set $\bv\in\partial h(\by)$ are called subgradients, and the inequality defining the set is called the subgradient inequality. A convex function is called \emph{proper} if it never attains the value $-\infty$. 
\begin{definition}\label{def:deltaSubgradient}
Let $\delta\geq 0$. For a convex function $h:\scrE\to(-\infty,+\infty]$, the $\delta$-subdifferential $\partial_{\delta}h(\by)$ the set of vectors $v\in\scrE^{\ast}$ satisfying 
\[
h(\tilde{\by})\geq h(\by)-\delta+\inner{v,\tilde{\by}-\by}\qquad \forall \tilde{\by}\in\scrE.
\]
\end{definition}
Note that the above definition reduces naturally to the convex subdifferential by setting $\delta=0$. 
\begin{definition}
The proximal operator of a closed convex and proper function $g:\scrE\to(-\infty,\infty]$ is defined by  
\begin{equation}\label{eq:proxop}
\prox_{\alpha g}(\bx)\eqdef\argmin_{\bu\in\scrE}\{g(\bu)+\frac{1}{2\alpha}\norm{\bu-\bx}^{2}\}
\end{equation}
\end{definition}
The prox-operator is always 1-Lipschitz (non-expansive) \cite{BauCom16}.
We also make use of the Pythagorean identity on the Euclidean space $\scrE$ with inner product $\inner{B\cdot,\cdot}$:
\begin{equation}\label{eq:Pythagoras}
2\inner{\by-\bu,B(\bx-\by)} =\norm{\bx-\bu}^{2}-\norm{\bx-\by}^{2}-\norm{\by-\bu}^{2}
\end{equation}
%
For $p\in[1,\infty]$, let $L^{p}(\Omega,\scrF,\Pr;\scrE)$ be the set of all random variables for which the integral 
$\Ex_{\Pr}\left[\abs{f}^{p}\right]\eqdef \int_{\Omega}\abs{f(\omega)}^{p}\dif\Pr(\omega)$ exists and is finite. This is a Banach space with norm $\abs{f}_{p}\eqdef \left(\Ex_{\Pr}\left[\abs{f}^{p}\right]\right)^{1/p}$.
\section{Problem Formulation}
\label{sec:Problem}
%

We denote by $(\scrX,\norm{\cdot}_{\scrX})$ and $(\scrY,\norm{\cdot}_{\scrY})$ finite dimensional Euclidean vector spaces, with dual spaces $(\scrX^{\ast},\norm{\cdot}_{\scrX^{\ast}}),(\scrY^{\ast},\norm{\cdot}_{\scrY^{\ast}})$. Let $(\Omega_{0},\scrA,\Pr_{0})$ be a complete probability space, carrying random elements $\xi_{1} \in L^{0}(\Omega_{0},\scrA_{0},\Pr_{0};\Xi_{1})$ and $\xi_{2}\in L^{0}(\Omega_{0},\scrA_{0},\Pr_{0};\Xi_{2})$ taking values in a measurable space $(\Xi_{i},\scrB(\Xi_{i})),i=1,2$. We define $\xi(\omega)\eqdef (\xi_{1}(\omega),\xi_{2}(\omega))$, and denote the distribution of this random element as $\measP_{\xi}\eqdef \Pr_{0}\circ\xi^{-1}$. Accordingly, the marginal distributions are defined as $\measP_{\xi_1}(A)\eqdef \measP_{\xi}(A\times\Xi_{2})$ and $\measP_{\xi}(\Xi_{1}\times B)\eqdef \measP_{\xi_2}(B)$ for $A\in\scrB(\Xi_{1})$ and $B\in\scrB(\Xi_{2})$, respectively.
 
\subsection{The hyperobjective program}
In problem \eqref{eq:SBL}, the variable $\by\in\scrY$ (i.e. the learning paramters)  is chosen before the event $\omega$ is realized, whereas $\bx$ is a decision variable (i.e. the model parameters) that is implemented just-in-time, given $\by\in\scrY$ and the realization $\xi_{2}(\omega)\in\Xi$. A solution of the lower-level optimization problems constitutes therefore of a feedback mapping $\bx^{\ast}(\cdot,\xi_{2})\in L^{\infty}(\scrY;\scrX)$, satisfying a measurability property with respect to the noise variable:
\[
\omega\mapsto \bx^{\ast}(\by,\xi_{2}(\omega))\in L^{0}(\Omega,\scrA_{0},\Pr_{0};L^{\infty}(\scrY;\scrX)).
\]
In particular, by the Doob-Dynkin Lemma, the mapping $\bx^{\ast}(\by,\xi_{2}(\cdot))$ is $\sigma(\xi_{2})$-measurable, for all $\by\in\scrY$. The following standing assumption shall apply throughout the paper.
\begin{assumption} 
 $r_{1}:\scrY\to(-\infty,\infty]$ is a closed convex and proper function.
 \end{assumption}
 \begin{assumption}
 \label{ass:Caratheodory}
 $F:\scrX\times\Xi_{1}\to\R$ is a Carath\'eodory function: 
 \begin{itemize}
 \item[(a)] $\xi_{1}\mapsto F(\bx,\xi_{1})$ is $\sigma(\xi_{1})$-measurable for every $\bx\in\scrX$;
 \item[(b)] $x\mapsto F(x,\xi_{1})$ is continuous for almost every $\xi_{1}\in\Xi_{1}$.
 \end{itemize}
 \end{assumption}
 
\begin{assumption}
\label{ass:UL}
The function $\bx\mapsto \Ex_{\Pr_{0}}[F(\bx,\xi_{1})]$ is finite for all $\bx\in\scrX$. There exists a positive valued random variable $\lip_{0}(F(\cdot,\xi_{1})):\Omega\to(0,\infty)$ such that $\abs{\lip_{0}(F(\cdot,\xi_{1}))}_{1}<\infty$, and for all $\bx_{1},\bx_{2}\in\scrX$ it holds that 
\begin{equation}\label{eq:FLipschitz}
\abs{F(\bx_{1},\xi_{1})-F(\bx_{2},\xi_{1})}\leq\lip_{0}(F(\cdot,\xi_{1}))\norm{\bx_{1}-\bx_{2}}_{\scrX}. 
\end{equation}
\end{assumption} 
  %

Assumption \ref{ass:UL} implies that $\bx\mapsto f(\bx)\eqdef \Ex_{\Pr_{1}}[F(\bx,\xi_{1})]$ is Lipschitz continuous \cite[][Thm. 7.44]{ShaDenRus09}, with Lipschitz constant $\lip_{0}(f)=\abs{\lip_{0}(F(\cdot,\xi_{1}))}_{1}$. In particular, the function $\bx\mapsto f(\bx)$ is measurable. 
\begin{assumption}
\label{ass:LL}
 $r_{2}:\scrX\to(-\infty,\infty]$ is proper, closed and convex. For all $\by\in\dom(r_{1})$, the function $\bx\mapsto g_{2}(\bx,\by,\xi_{2})$ is continuously differentiable and convex.
 \end{assumption}
 \begin{assumption}
 \label{ass:ULL}
 Let $(\by,\xi_{2})\in\Int\dom(r_{1})\times\Xi_{2}$ be given. The parameterized variational inequality 
 \begin{equation}\label{eq:VI}
\text{ Find $\bx\in\scrX$ such that }0\in \nabla_{x}g(\bx,\by,\xi_{2})+\partial r_{2}(\bx) 
 \end{equation}
 has a unique solution $\bx^{\ast}(\by,\xi_{2})$, enjoying the following properties:
 \begin{itemize}
 \item[(S.1)] $\xi_{2}\mapsto \bx^{\ast}(\by,\xi_{2})$ is measurable, uniformly in $\by\in\Int\dom(r_{1})$;
 \item[(S.2)] $\by\mapsto \bx^{\ast}(\by,\xi_{2})$ is Lipschitz continuous on $\Int\dom(r_{1})$, for almost all $\xi_{2}\in\Xi_{2}$. 
 \end{itemize}
 \end{assumption}
Combining Assumptions \ref{ass:UL} and \ref{ass:ULL}, we can define the stochastic hyperobjective
\begin{equation}\label{eq:implicit}
H:\scrY\times\Xi\to\R,\quad  (\by,\xi)\mapsto H(\by,\xi)\eqdef F(\bx^{\ast}(\by,\xi_{2}),\xi_{1}).
 \end{equation}
Note that $H(\cdot,\xi)\in\bC^{0,0}(\scrY)$. In order to bound the variance of our gradient estimator, we need an a-priori assumption on the integrability of the random Lipschitz modulus.
\begin{assumption}
\label{ass:LipH}
We assume that $\abs{\lip_{0}(H(\cdot,\xi))}_{2}<\infty$. 
\end{assumption}
Thanks to the inherited measurability, we can leverage Fubini's theorem to obtain $h(\by)\eqdef\Ex_{\Pr}[H(\by,\xi)]=\int_{\Xi_{2}}f(\bx^{\ast}(\by,w_{2}))\dif\measP_{\xi_{2}}(w_{2})$.
The fact that $f\in\bC^{0,0}(\scrY)$ combined with (S.2) allows us to conclude $h\in\bC^{0,0}(\scrY)$. 

Absorbing the lower level solution into the upper level, we arrive at the reduced formulation of the upper level optimization problem 
\begin{equation}\label{eq:ImplicitProg}
\Psi^{\Opt}\eqdef\inf_{\by\in\scrY}\{\Psi(\by)\eqdef h(\by)+r_{1}(\by)\},
\end{equation}
 which is commonly known in bilevel optimization as the \emph{hyperobjective optimization problem}. 
 
\subsection{Approximate stationarity conditions}
\label{sec:stationary}
The hyperobjective program \eqref{eq:ImplicitProg} is a non-convex and non-smooth optimization problem, involving a Lipschitz continuous function $\by\mapsto h(\by)$, and a convex composite term $\by\mapsto r_{1}(\by)$. As is typical in non-convex optimization, our aim is to localize a specific class of approximate stationary points, as we are about to define in this section. For a locally Lipschitz function $h:\scrY\to \R$, the generalized directional derivative \cite{Clar90} of $h$ at $\by\in\scrY$ in direction $\bu\in\scrY$ is defined as 
\[
h^{\circ}(\by;\bu)=\lim\sup_{\by'\to\by,t\to 0^{+}}\frac{h(\by'+t\bu)-h(\by')}{t}
\]
The Clarke subdifferential of $h$ at $\by$ is the set 
\[
\partial_{\rm C}h(\by)\eqdef\{s\in\scrY^{\ast}\vert h^{\circ}(\by,\bu)\geq \inner{s,\bu}\;\forall \bu\in\scrY\}.
\]
The primary goal of non-smooth non-convex optimization is the search for stationary points. A point $\by\in\scrY$ is called (Clarke)-\emph{stationary} for $\Psi=h+r$ if the inclusion 
\[
0\in\partial_{C}h(\by)+\partial r_{1}(\by)
\]
is satsified.
\begin{definition}
Given $\eps>0$, a point $\by^{\ast}\in\scrY$ is called an $\eps$-stationary point of \eqref{eq:ImplicitProg} if 
\begin{equation}
\dist(0,\partial_{C}\Psi(\by^{\ast}))\leq \eps. 
\end{equation}
\end{definition}
Recently, a series of papers challenged the question whether optimization algorithms are able to identify $\eps$-stationary points in finite time. \cite{Zhang:2020aa} provided a definite negative answer to this question, by demonstrating that no first-order method is able to identify $\eps$-stationary points in finite time. Therefore, we will content ourselves with a more modest stationarity notion.
\begin{definition}[\cite{Gold77}]
\label{def:Gold}
For any $\delta>0$, the Goldstein $\delta$-subdifferential of $h$ at $\by\in\scrY$ is the set
\begin{equation}\label{eq:Gold}
\partial^{\delta}_{\rm G}h(\by)\eqdef \conv\left(\bigcup_{\tilde{y}\in\B(\by,\delta)}\partial_{\rm C}h(\tilde{y})\right).
\end{equation}
\end{definition}
We employ the Goldstein subdifferential for relating the stationarity measures of a smoothed auxiliary model, with stationarity with respect to the original problem. As such, our proposal of an approximate stationary point combines the definitions of \cite{Davis:2019aa,DavDru19} for stochastic subgradient methods, and \cite{Lei:2024aa} for zeroth-order methods. 
\begin{definition}
\label{def:GSstationarypoint}
For any $(\eps,\delta)>0$, we call a random variable $\by^{\ast}\in L^{0}(\Omega,\scrF,\Pr;\scrY)$ an $(\eps,\delta)$-stationary point of \eqref{eq:Psi} if
\begin{equation}\label{eq:GoldStationary}
\Ex\left[\dist\left(\by^{\ast},\{\by\vert \dist(0,\partial^{\delta}_{G}h(\by)+\partial r_{1}(\by))^{2}\leq\eps\}\right)^{2}\right]\leq\eps\,. 
\end{equation}
\end{definition}


\section{Derivative free randomized proximal gradient method}
\label{sec:ZeroOrder}
\subsection{Gaussian Smoothing of the implicit function}
\label{sec:smoothing}
%
To simplify the notation, we write $\norm{\bu}_{\scrY}\equiv \norm{\bu}=\sqrt{\inner{B\bu,\bu}}$, given the Riesz mapping $B=B^{\ast}\succ 0$ from $\scrY$ to $\scrY^{\ast}$. We denote the dimension of the Euclidean space $\scrY$ by $n$. The $n$-dimensional Lebesgue measure on $(\scrY,\scrB(\scrY))$ is denoted by $\Leb_{\scrY}$, and we typically write $\dif\by$, instead of $\dif\Leb_{\scrY}(\by)$. We define the Gaussian Lebesgue density on $(\scrY,\scrB(\scrY),\Leb_{\scrY})$ as 
$$
\pi_{\eta}(\bz\vert \by)\eqdef \frac{\sqrt{\det(B)}}{(2\pi)^{n/2}\eta^{n}}\exp\left(-\frac{1}{2\eta^{2}}\norm{\bz-\by}^{2}\right)=\frac{1}{\Upsilon\eta^{n}}\exp\left(-\frac{1}{2\eta^{2}}\norm{\bz-\by}^{2}\right),\quad \Upsilon\eqdef\frac{(2\pi)^{n/2}}{\sqrt{\det(B)}}. 
$$
Given a function $h:\scrY\to\R$ and a positive parameter $\eta>0$, we define the Gaussian smoothing of $h$ as the convolution 
\begin{equation}\label{eq:smoothing}
h_{\eta}(\by)\eqdef (h\circledast \pi_{\eta})(\by)=\int_{\scrY}h(\bz)\pi_{\eta}(\bz\vert\by)\dd\bz. 
\end{equation}
Let us introduce an independent probability space $(\Omega_{1},\scrA_{1},\Pr_{1})$. We say $U:(\Omega_{1},\scrA_{1})\to(\scrY,\scrB(\scrY))$ is a standard Gaussian random variable on $\scrY$, denoted as $U\sim \Normal(0,\Id_{\scrY})$, if $\Pr_{1}\circ U^{-1}$ admits the density $\pi_{1}(\cdot\vert\0)\equiv\pi$ on $\scrY$ with respect to $\Leb_{\scrY}$. Via the change of variables $\bz=\by+\eta\bu$, we can rewrite the above integral as 
\[
h_{\eta}(\by)=\int_{\scrY}h(\by+\eta\bu)\pi(\bu)\dd\bu=\Ex_{\Pr_{1}}[h(\by+\eta U)].
\]
For $\eta>0$, the function $\by\mapsto h_{\eta}(\by)$ is differentiable and $\eta>0$ plays the role of a smoothing parameter. Using the expression above, we immediately deduce the formula for the gradient as 
\begin{equation}\label{eq:GradImplicit}
\nabla h_{\eta}(\by)=\Ex_{\Pr_{1}}\left[\frac{h(\by+\eta U)}{\eta}BU\right]=\Ex_{\Pr_{1}}\left[\frac{h(\by+\eta U)-h(\by)}{\eta}BU\right]
\end{equation}
Specifically, we leverage upon the work \cite{Nesterov:2017wm}, and use the following estimates.\footnote{For being self-contained, we provide proofs of these facts in Appendix \ref{app:Smoothing}.}
\begin{lemma}\label{lem:Liph}
Let $h\in\bC^{0,0}(\scrY)$. Then $h_{\eta}\in\bC^{0,0}(\scrY)$ and $\lip_{0}(h_{\eta})\leq \lip_{0}(h)$ for all $\eta>0$. 
\end{lemma}
\begin{lemma}[\cite{Nesterov:2017wm}, Theorem 1]
\label{lem:boundfunctions}
Let $h\in\bC^{0,0}(\scrY)$ and $\eta>0$. Then for all $\by\in\scrY$ it holds
\[
\abs{h_{\eta}(\by)-h(\by)}\leq\eta \lip_{0}(h)\sqrt{n}\,.
\]
\end{lemma}
\begin{lemma}\label{lem:boundGradSmoothed}
Let $h\in\bC^{0,0}(\scrY)$ and $\eta>0$. Then $h_{\eta}\in\bC^{1,1}(\scrY)$ with $\lip_{1}(h_{\eta})=\frac{\sqrt{n}}{\eta}\lip_{0}(h)$. Moreover, for all $\by\in\scrY$, there holds
\begin{equation}\label{eq:boundGradSmoothed}
\norm{\nabla h_{\eta}(\by)}^{2}_{\ast}\leq  \lip_{0}(h)^{2}(4+n)^{2}.
\end{equation}
\end{lemma}

In the convex case, we report a classical relation between the gradients of the Gaussian smoothed function and the $\delta$-subdifferential. 
\begin{lemma}[\cite{Nesterov:2017wm}, Theorem 2]
\label{lem:deltagradient}
If $h\in\bC^{0,0}(\scrY)$ and convex, then, for all $\by\in\scrY$ and $\eta>0$, we have
\begin{equation}
\nabla h_{\eta}(\by)\in\partial_{\delta}h(\by),\quad\text{ for }
\delta=\eta\lip_{0}(h)\sqrt{n}
\end{equation}
where $partial_{\delta}h$ is the $\delta$-subdifferential (cf. Defintion \ref{def:deltaSubgradient}). 
\end{lemma}
The next proposition establishes a quantitative connection between the gradients of the smoothed function $h_{\eta}$ and the Goldstein $\delta$-subgradient. This is the key tool to relate complexity estimates of the smoothed objective with the original, unsmoothed, objective.  
\begin{proposition}[\cite{Lei:2024aa},Theorem 3.6]
\label{prop:epssmoothedgrad}
Let $h\in\bC^{0,0}(\scrY)$ and $\scrD\subset\scrY$ a convex compact set. Then, for all $\delta>0$ and for all $\eps>0$, it holds that 
$$
\nabla h_{\eta}(\by)\in\partial^{\delta}_{\rm G}h(\by)+\eps\B_{\scrY}\qquad\forall \eta\in(0,\bar{\eta}],\forall \by\in\scrD.
$$
where $\B_{\scrY}$ denotes the unit ball in $\scrY$, $\bar{\eta}\eqdef \min\{1,\delta/\Gamma\},\Gamma\eqdef \left[-n \scrW_{-1}\left(\frac{-\nu^{2/3}}{2\pi e}\right)\right]^{1/2}$ and $\nu\eqdef\min\{\frac{\eps}{4\lip_{0}(h)},(2\pi)^{n/2}-\frac{1}{2}\}$. $\scrW_{-1}$ is the negative branch of the Lambert $W$-function, i.e.~of the inverse of $x\mapsto xe^x$, $x\in\R$.
\end{proposition}
\begin{remark}
Since $\nu\leq (2\pi)^{\frac{n}{2}}-\frac{1}{2}$, we have $\frac{\nu^{2/n}}{2\pi e}<\frac{1}{e}$, and hence $\scrW_{-1}\left( \frac{\nu^{2/n}}{2\pi e}\right)<-1$. Thus, $\Gamma\in(\sqrt{n},\infty)$. \close
\end{remark}

\subsection{Zeroth-order gradient estimator of the implicit function} 
%

The first step in our construction is the design of a zeroth-order gradient estimator. This requires a solution to the lower-level problem. We discuss two different settings. First, we consider the case in which the solution of the lower level problem is available exactly. This is a very common assumption in stochastic bilevel optimization; see e.g.~\cite{BilevelLearningIEEE,Cui:2023aa,Cui:2022aa}, as well as \cite{Ehrhardt:2021ab} for inverse problems. We then relax this assumption by allowing for controllable implementation errors in the lower level solution. This scenario is more realistic, but also more challenging since the inexact model introduces an additional bias in the stochastic gradient estimator. We account for this additional difficulty by presenting two different complexity estimates, one for the exact and one for the inexact case, respectively. 

\subsection{Exact lower level solution}
Consider the implicit function $h:\scrY\to\R$ given by $h(\by)=\Ex_{\Pr}[H(\by,\cdot)]$, where $H(\by,\xi)=F(\bx^{\ast}(\by,\xi_{2}),\xi_{1})$ is the hyperobjective, defined in \eqref{eq:implicit}. We have $h\in\bC^{0,0}(\scrY)$, so that its Gaussian smoothing with parameter $\eta>0$ satisfies $h_{\eta}\in\bC^{1,1}(\scrY)$.  Let $\bu\in\scrY$ represent a direction and $\delta>0$ a parameter. We define the finite-difference estimator
\[
\hat{\nabla}_{(\bu,\eta)} H(\by,\xi)\eqdef \frac{H(\by+\eta\bu,\xi)-H(\by,\xi)}{\eta}B\bu=\frac{F(\bx^{\ast}(\by+\eta\bu,\xi_{2}),\xi_{1})-F(\bx^{\ast}(\by,\xi_{2}),\xi_{1})}{\eta}B\bu. 
\]
If $\bu^{(m)}=\{\bu^{(1)},\ldots,\bu^{(m)}\}$ is an $m$-tuple of directions in $\scrY$ and $\xi^{(m)}=\{\xi^{1},\ldots,\xi^{m}\}$ are $m$-i.i.d copies of the random variable $\xi$, then we define the random gradient estimator, based on finite differences of the subsampled hyperobjective: 
\begin{equation}\label{eq:V}
V_{\eta}(\by,\bu^{(m)},\xi^{(m)})\eqdef \frac{1}{m}\sum_{i=1}^{m}\hat{\nabla}_{(\bu^{i},\eta)}H(\by,\xi^{i}).
\end{equation}
To realize this estimator on a sufficiently large common probability space, we build the typical product space enlargement $(\Omega,\scrA,\Pr)=(\Omega_0\times\Omega_1,\scrA_{0}\otimes\scrA_{1},\Pr_{0}\times\Pr_{1})$. On this extended setup, we abuse notation and identify the random element $\xi$ and $U$ as measurable functions on $(\Omega,\scrA)$ by means of the following notational convention:
\[
\xi(\omega)=\xi(\omega_{0})\text{ and }U(\omega)=U(\omega_{1})\qquad\forall \omega\in\Omega.
\]
Let $U^{(m)}\eqdef (U^{1},\ldots,U^{m})$ be an iid random sample of Gaussian $\scrY$-valued random vectors and $\xi^{(m)}\eqdef(\xi^{1},\ldots,\xi^{m})$ an iid sample of $\xi$, assumed to be statistically independent of each other. Define the random estimator 
\begin{equation}\label{eq:batchestimator}
\hat{V}_{\eta,m}(\by,\omega)\eqdef V_{\eta}(\by,U^{(m)}(\omega),\xi^{(m)}(\omega))\qquad\forall \omega\in\Omega.
\end{equation}
Given a positive smoothing parameter $\eta>0$, we are iteratively solving the stochastic composite optimization problem 
\begin{equation}\label{eq:aux}
\Psi^{\Opt}_{\eta}\eqdef\min_{y\in\scrY}\{\Psi_{\eta}(\by)\eqdef h_{\eta}(\by)+r_{1}(\by)\}\quad\text{with }h_{\eta}(\by)=\Ex_{\Pr}[F(\bx^{\ast}(\by+\eta U,\xi_{2}),\xi_{1})].
\end{equation}
The smooth part of this composite minimization problem is the Gaussian smoothing of the hyperobjective $h$, and $r_{1}$ is a closed convex and proper regularizing term.

\subsection{Inexact lower level solution}
We now define a relaxation of the stochastic oracle, allowing for computational errors in the lower level solution. 
\begin{definition}[Inexact lower level solution]\label{def:inexact_lower}
Given $p\geq 2$ and $\beta\geq 0$, we call a mapping $\bx^{\beta}\in L^{\infty}(\scrY\times\Xi;\scrX)$ an $\beta$-optimal solution of the lower level problem if
\begin{equation}\label{eq:inexact}
\Ex\left[\norm{\bx^{\beta}(\by,\xi)-\bx^{\ast}(\by,\xi)}^{p}_{\scrX}\right]^{1/p}\leq \beta.
\end{equation}
\end{definition}
\begin{remark}
We note that an inexact solution can readily be obtained by embedding our main iteration in a double-loop algorithmic strategy in which the inner loop is some fast solver that returns an approximate solution to the lower level problem, for fixed parameters $(\by,\xi)$. The exact formulation of such an inner loop solver should be adapted to the nature of the lower level optimization problem. Since we are aiming for a general-purpose methodology, we do not specify the explicit modelling of such an algorithm, and instead treat this feature of our method as a black box. Double loop schemes are very popular solution strategy in stochastic bilevel optimization (cf. \cite{Cui:2022aa,Kwon:2023aa} and references therein). Inexactness of lower level solutions in bilevel optimization has been investigated in \cite{Ehrhardt:2021ab,Ehrhardt:2024aa} in deterministic regimes. Our notion takes into consideration the potential noisy nature of the data. 
\end{remark}
Given the inexact lower level solution mapping, we accordingly define the inexact hyperobjective as
\[
H^{\beta}(\by,\xi)\eqdef F(\bx^{\beta}(\by,\xi_{2}),\xi_{1}).
\]
The resulting biased random gradient estimator is given by 
\[
\hat{\nabla}_{(\bu,\eta)} H^{\beta}(\by,\xi)\eqdef \frac{H^{\beta}(\by+\eta\bu,\xi)-H^{\beta}(\by,\xi)}{\eta}B\bu=\frac{F(\bx^{\beta}(\by+\eta\bu,\xi_{2}),\xi_{1})-F(\bx^{\beta}(\by,\xi_{2}),\xi_{1})}{\eta}B\bu,
\]
and replace the multi-point random gradient estimator by
\begin{equation}\label{eq:Vinexact}
V^{\beta}_{\eta}(\by,\bu^{(m)},\xi^{(m)})\eqdef \frac{1}{m}\sum_{i=1}^{m}\hat{\nabla}_{(\bu^{i},\eta)}H^{\beta}(\by,\xi^{i}).
\end{equation}
As in the exact case, in order to reduce notational clutter, we will adopt the simplified notation $\hat{V}^{\beta}_{\eta,m}(\by,\omega)\eqdef V^{\beta}_{\eta}(\by,U^{(m)}(\omega),\xi^{(m)}(\omega))$ for the multi-point random gradient estimate based on the zeroth-order oracle.

\subsection{The algorithmic scheme}

 Since $h_{\eta}\in\bC^{1,1}(\scrY)$, we are in the classical proximal-gradient framework, which is defined in terms of the fixed point iteration 
\[
\bar{\by}^{+}=T_{\eta,t}(\by)\eqdef \prox_{tr_{1}}(\by-tB^{-1}\nabla h_{\eta}(\by)),
\]
where $t\in[0,\infty)$ is a step size parameter. 
The \emph{prox-gradient mapping} is the operator $\scrG_{\eta,t}:\scrY\to\scrY$ defined by 
\begin{equation}\label{eq:gradient}
\scrG_{\eta,t}(\by)\eqdef \frac{1}{t}(\by-T_{\eta,t}(\by)).
\end{equation}
Since we have no direct access to the gradient $\nabla h_{\eta}(\by)$, we define a stochastic approximation using the operator $P_{t}:\scrY\times\scrY^{\ast}\to\scrY$ defined by 
\begin{equation}\label{eq:RandProx}
P_{t}(\by,v)\eqdef\prox_{tr_{1}}(\by-tB^{-1}v)\qquad\forall (\by,v)\in\scrY\times\scrY^{\ast}. 
\end{equation}
Clearly, $P_{t}(\by,\nabla h_{\eta}(\by))=T_{\eta,t}(\by)$ for $\by\in\scrY$. The stochastic analogue to the prox-gradient mapping is the random operator $\tilde{\scrG}_{\eta,t}:\scrY\times\Omega\to\scrY$,
\begin{equation}\label{eq:randomgradient}
\tilde{\scrG}_{\eta,t}(\by,\omega)\eqdef\frac{1}{t}\left(\by-P_{t}(\by,\hat{V}_{\eta,m}(\by,\omega))\right).
\end{equation}
Note that if $r_{1}=0$, then $\tilde{\scrG}_{\eta,t}(\by,\omega)=\hat{V}_{\eta,m}(y,\omega)$ for all $(\by,\omega)\in\scrY\times\Omega$. 
\begin{algorithm}
\caption{Derivative-free approximate prox-grad algorithm}
\label{alg:proxderivativefree}
\begin{algorithmic}
\Require{ $\by_{0}\in\dom(r_{1})$ and $N\in\N$. Let $(\alpha_{k})_{k=0}^{N-1}\subset(0,\infty)$, $(\eta_k)_{k=0}^{N-1}$, and $(m_{k})_{k=1}^{N}$ be sequences in $\N$. }
\For{$k=0,\ldots,N-1$}
\State Generate $\hat{V}_{k+1}\eqdef V_{\eta_{k}}(y_{k},U^{(m_{k+1})},\xi^{(m_{k+1})})$ by \eqref{eq:V} (exact lower level) or \eqref{eq:Vinexact} (inexact lower level);
\State Update $\by_{k+1}=P_{\alpha_{k}}(\by_{k},\hat{V}_{k+1})$
\EndFor
\end{algorithmic}
\end{algorithm}

 \subsection{Properties of the gradient estimator with exact lower level solutions}
 In this section we work out some a-priori error estimates on the random gradient estimator \eqref{eq:V}. Whenever convenient, we suppress the dependence on $\omega$, and simply write  $\hat{V}_{\eta,m}(\by)\equiv V_{\eta}(\by,U^{(m)},\xi^{(m)})$. The first Lemma shows that our random estimator is unbiased in terms of the gradient operator of the smoothed function $h_{\eta}$. 
 \begin{lemma}\label{lem:errorbound}
 For all $\by\in\scrY$, we have $\Ex_{\Pr}[\hat{V}_{\eta,m}(\by)]=\nabla h_{\eta}(\by)$ and
 \[
\Ex_{\Pr}\left[\norm{\hat{V}_{\eta,m}(\by)}^{2}_{\ast}\right]-\|\nabla h_\eta(\by)\|_\ast^2\leq \frac{\cs^{2}}{m}\eqdef  \frac{(4+n)^{2}\abs{\lip_{0}(H(\cdot,\xi))}^{2}_{2}}{m}.
\]
\end{lemma}
\begin{proof}
See Appendix \ref{app:1}.
\end{proof}
We define the error process 
\begin{equation}
\Delta W_{\eta,m}(\by,\omega)\eqdef \hat{V}_{\eta,m}(\by,\omega)-\nabla h_{\eta}(\by) \qquad\forall (\by,\omega)\in\scrY\times\Omega.
\end{equation}
 An immediate corollary of Lemma \ref{lem:errorbound} is that the error process defines essentially a martingale difference sequence:
\begin{align}
&\Ex_{\Pr}[\Delta W_{\eta,m}(\by)]=0,\text{ and} \\ 
&\Ex_{\Pr}\left[\norm{\Delta W_{\eta,m}(\by)}_{\ast}^{2}\right]=\Ex_{\Pr}\left[\norm{\hat{V}_{\eta,m}(\by)}_{\ast}^{2}\right]-\norm{\nabla h_{\eta}(\by)}^{2}_{\ast}\leq \frac{\cs^{2}}{m}\,.
\end{align}
Moreover, the error process can be used to estimate the prox-gradient mapping as follows:
\begin{lemma}\label{lem:RelationGradient}
We have 
\begin{equation}
\norm{\scrG_{\eta,t}(\by)}^{2}\leq 2\norm{\tilde{\scrG}_{\eta,t}(\by)}^{2}+2\norm{\Delta W_{\eta,m}(\by)}^{2}_{\ast}\qquad \text{a.s.} 
\end{equation}
\end{lemma}
\begin{proof}
Using the non-expansiveness of the prox-operator, we obtain
\begin{align*}
\norm{\scrG_{\eta,t}(\by)}^{2}&=\norm{\frac{1}{t}[\by-P_{t}(\by,V_{\eta,m}(\by))]+\frac{1}{t}[P_{t}(\by,V_{\eta,m}(\by))-T_{\eta,t}(\by)]}^{2}\\ 
&\leq 2\norm{\tilde{\scrG}_{\eta,t}(\by)}^{2}+\frac{2}{t^{2}}\norm{P_{t}(\by,V_{\eta,m}(\by))-T_{\eta,t}(\by)}^{2}\\ 
&\leq 2\norm{\tilde{\scrG}_{\eta,t}(\by)}^{2}+2\norm{B^{-1}(V_{\eta,m}(\by)-\nabla h_{\eta}(\by))}^{2}= 2\norm{\tilde{\scrG}_{\eta,t}(\by)}^{2}+2\norm{\Delta W_{\eta,m}(\by))}^{2}_{\ast}.
\end{align*}
\end{proof}

\subsection{Properties of the gradient estimator with inexact lower level solutions}
The inexactness of the solution of the lower-level problem will have its trace on the variance of the random estimator. The bias can be described by means of the following error decomposition. 
\begin{lemma}\label{lem:biasestimate}
For all $\by\in\scrY$ and $\beta>0$, it holds
\begin{equation}\label{eq:gradbiased}
\Ex_{\Pr}[\hat{V}^{\beta}_{\eta,m}(\by)]=\nabla h_{\eta}(\by)+ \frac{1}{m}\sum_{i=1}^m \Ex_{\Pr}\Big[ \frac{F(\bx^{\beta}(\by+\eta U^i,\xi^{i}),\xi^{i}_{1})-F(\bx^\ast(\by+\eta U^i,\xi^{i}),\xi^{i}_{1})}{\eta} BU^{i}\Big] \,,
\end{equation}
and
\begin{equation}\label{eq:biasestimate}
\begin{split}
 \frac{1}{m}\sum_{i=1}^m \norm{\Ex_{\Pr}\Big[ \frac{F(\bx^{\beta}(\by+\eta U^i,\xi^{i}),\xi^{i}_{1})-F(\bx^\ast(\by+\eta U^i,\xi^{i}_{2}),\xi^{i}_{1})}{\eta} BU^{i}\Big]}_{\ast} \\
\leq \frac{\sqrt{n}\abs{\lip_{0}(F(\cdot,\xi_{1})}_{1}}{\eta}\Ex_{\Pr}\Big[\norm{x^{\beta}(y+\eta U,\xi_{2})-x^{\ast}(y+\eta U,\xi_{2})}_{\scrX}^{p}\Big]^{\frac{1}{p}}.
\end{split}
\end{equation}
\end{lemma}
\begin{proof}
See Appendix \ref{app:1}.
\end{proof}
Let $(y_{k})_{k}$ be the stochastic process whose sample paths are generated via Algorithm \ref{alg:proxderivativefree}. The natural filtration associated with this process is $\scrF_{k}\eqdef\sigma(y_{1},\ldots,y_{k})$. Along the sample paths of this process, we can perform the following error decomposition of the random gradient estimators:  
\begin{equation}\label{eq:errordecomp}
\hat{V}^{\beta}_{k+1}=\hat{V}_{k+1}-a_{k+1}+b_{k+1},
\end{equation}
with 
\begin{align*}
&a_{k+1}\eqdef \frac{1}{m_{k+1}}\sum_{i=1}^{m_{k+1}}\frac{F(\bx^{\beta_{k}}(y_{k},\xi^{i}_{2,k+1}),\xi_{1,k+1}^{i})-F(\bx^{\ast}(y_{k},\xi^{i}_{2,k+1}),\xi_{1,k+1}^{i})}{\eta} BU^{i}_{k+1},\\
&b_{k+1}\eqdef\frac{1}{m_{k+1}}\sum_{i=1}^{m_{k+1}}\frac{F(\bx^{\beta_{k}}(y_{k}+\eta U^{i}_{k+1},\xi^{i}_{2,k+1}),\xi_{1,k+1}^{i})-F(\bx^{\ast}(y_{k}+\eta U^{i}_{k+1},\xi^{i}_{2,k+1}),\xi_{1,k+1}^{i})}{\eta} BU^{i}_{k+1}
\end{align*}
Note that $\Ex(a_{k+1}\vert\scrF_{k})=0$, and we can derive a bound in $L^{2}(\Pr)$ as the following Lemma shows.
\begin{lemma}\label{lem:errorboundinexact}
Let be $p>2$ the exponent from Definition \ref{def:inexact_lower}. There exists a 
constant $C_{F}>0$, such that 
\begin{align}
&\Ex\left[\norm{a_{k+1}}^{2}_{\ast}\vert\scrF_{k}\right]\leq C_{F}\frac{\beta^{2}_{k}}{\eta^{2}},\text{ and }\Ex\left[\norm{b_{k+1}}^{2}_{\ast}\vert\scrF_{k}\right]\leq C_{F}\frac{\beta^{2}_{k}}{\eta^{2}},
\end{align}
\end{lemma}
\begin{proof}
See Appendix \ref{proof:lem4.10}.
\end{proof}

\section{Complexity analysis for the Non-Convex case}
\label{sec:Nonconvex}
%

\subsection{Exact lower level solution}
We begin our convergence analysis in the non-convex setting, focusing on cases where the lower-level problem can be solved exactly. Our first Lemma provides an estimate on the per-iteration function progress in terms of the smoothed hyperobjective $\Psi_{\eta}$. 
\begin{lemma}\label{lem:Psiprogress}
Consider the sequence $(y_{k})_{k\in\N}$ generated by Algorithm \eqref{alg:proxderivativefree} with gradient estimator \eqref{eq:V}. Then, for all $\eta>0$, we have 
\begin{equation}\label{eq:Psiprogress}
\begin{split}
\Psi_{\eta}(y_{k+1})-\Psi_{\eta}(y_{k}) \leq& -\alpha_{k}\norm{\tilde{\scrG}_{\eta,\alpha_{k}}(y_{k})}^{2}\left(1-\frac{\alpha_{k}\lip_{1}(h_{\eta})}{2}\right)+\alpha_{k}\inner{\Delta W_{k+1},\scrG_{\eta,\alpha_{k}}(y_{k})}+\alpha_{k}\norm{\Delta W_{k+1}}_{\ast}^{2}.
\end{split}
\end{equation}
\end{lemma}
\begin{proof}
See Appendix \ref{proof:Lemma5.1}.
\end{proof}
Set
\begin{align*}
&E_{k+1}\eqdef \norm{\Delta W_{k+1}}^{2}_{\ast}+\inner{\Delta W_{k+1},\scrG_{\alpha_{k}}(y_{k})}\text{ and } \\
&\Psi^{\Opt}_{\eta}\eqdef \min_{\by\in\scrY}\Psi_{\eta}(y).
\end{align*}
Summing \eqref{eq:Psiprogress} from $k=1,\ldots,N$, we obtain
\[
\sum_{k=1}^{N}\alpha_{k}\left(1-\frac{\lip_{1}(h_{\eta})\alpha_{k}}{2}\right)\norm{\tilde{\scrG}_{\eta,\alpha_{k}}(y_{k})}^{2}\leq \Psi_{\eta}(y_{1})-\Psi_{\eta}(y_{N+1})+\sum_{k=1}^{N}\alpha_{k}E_{k+1} \leq \Psi_{\eta}(y_{1})-\Psi_{\eta}^{\Opt}+\sum_{k=1}^{N}\alpha_{k}E_{k+1}.
\]
Let $\scrF_{k}\eqdef \sigma(y_{1},\ldots,y_{k})$ denote the natural filtration up to time $k$ of the process, so that 
\[
\Ex_{k}(E_{k+1})\eqdef \Ex[E_{k+1}\vert\scrF_{k}]=\Ex[\norm{\Delta W_{k+1}}^{2}_{\ast}\vert\scrF_{k}]\leq \frac{\cs^{2}}{m_{k+1}},\quad\text{a.s.\,.}
\]
Therefore, using the law of iterated expectations, we obtain
\begin{equation}\label{eq:MainNonconvex}
\Ex\left[\sum_{k=1}^{N}\alpha_{k}\left(1-\frac{\lip_{1}(h_{\eta})\alpha_{k}}{2}\right)\norm{\tilde{\scrG}_{\eta,\alpha_{k}}(y_{k})}^{2}\right]\leq \Psi_{\eta}(y_{1})-\Psi_{\eta}^{\Opt}+\sum_{k=1}^{N}\frac{\alpha_{k}\cs^{2}}{m_{k+1}}
\end{equation}
where $\cs\eqdef (4+n)\abs{\lip_{0}(H(\cdot,\xi))}_{2}$. This yields our first main result in this paper:
\begin{theorem}\label{th:nonconvex1}
Let $(y_k)_{k\in\N}$ be generated by Algorithm~\ref{alg:proxderivativefree} with gradient estimator \eqref{eq:V}. Let the step sizes $(\alpha_{k})_{k\in\N}$ be chosen such that $\alpha_{k}\in(0,2/\lip_{1}(h_{\eta})]$, with $\alpha_{k}<2/\lip_{1}(h_{\eta})$ for at least on $k\in\{1,\ldots,N\}$. On $(\Omega,\scrF,\Pr)$ define an independent random variable $\kappa:\Omega\to\{1,\ldots,N\}$ with probability mass function 
\begin{equation}\label{eq:lawkappa}
p(k)\eqdef\Pr(\kappa=k)\eqdef \frac{\alpha_{k}-\alpha_{k}^{2}\lip_{1}(h_{\eta})/2}{\sum_{t=1}^{N}(\alpha_{t}-\alpha^{2}_{t}\lip_{1}(h_{\eta})/2)},\quad k\in\{1,\ldots,N\}.
\end{equation}
Then 
\begin{equation}\label{eq:complexityGap}
\Ex\left[\norm{\tilde{\scrG}_{\eta,\alpha_{\kappa}}(y_{\kappa})}^{2}\right]\leq \frac{\Psi_{\eta}(y_{1})-\Psi_{\eta}^{\Opt}+\sum_{k=1}^{N}\frac{\alpha_{k}\cs^{2}}{m_{k+1}}}{\sum_{t=1}^{N}(\alpha_{t}-\alpha^{2}_{t}\lip_{1}(h_{\eta})/2)},
\end{equation}
where $\cs\eqdef(4+n)\abs{\lip_{0}(H(\cdot,\xi))}_{2}$.
\end{theorem}
\begin{proof}
Using  eq. \eqref{eq:MainNonconvex}, together with the observation that 
\[
\Ex\left[\norm{\tilde{\scrG}_{\eta,\alpha_{\kappa}}(y_{\kappa})}^{2}\right]=\sum_{k=1}^{N}\frac{\alpha_{k}-\alpha_{k}^{2}\lip_{1}(h_{\eta})/2}{\sum_{t=1}^{N}(\alpha_{t}-\alpha^{2}_{t}\lip_{1}(h_{\eta})/2)}\Ex\left[\norm{\tilde{\scrG}_{\eta,\alpha_{k}}(y_{k})}^{2}\right],
\]
the thesis follows.
\end{proof}
A few remarks are in order. First, due to the ratio $\frac{\alpha_{k}}{m_{k+1}}$, there is a trade-off between too aggressive step-sizes and the size of the mini-batches. In fact, consider the particular step-size policy $\alpha_{k}\leq \frac{1}{\lip_{1}(h_{\eta})}$. Then, it follows $\frac{\lip_{1}(h_{\eta})}{2}\alpha^{2}_{k}\leq \frac{1}{2}\alpha_{k}$. Therefore, the numerator in our complexity bound \eqref{eq:complexityGap} can be further bounded as 
\[
\Ex\left[\norm{\tilde{\scrG}_{\eta,\alpha_{\kappa}}(y_{\kappa})}^{2}\right]\leq \frac{\Psi_{\eta}(y_{1})-\Psi_{\eta}^{\Opt}+\sum_{k=1}^{N}\frac{\alpha_{k}\cs^{2}}{m_{k+1}}}{\sum_{t=1}^{N}(\alpha_{t}/2)}.
\]
This bound suggests choosing step sizes like $\alpha_{k}=\frac{2\beta}{\lip_{1}(h_{\eta})\sqrt{k}}$ with $\beta\in(0,1/2)$, and mini-batches $m_{k+1}=\ca \sqrt{k}$, with $\ca>0$, to obtain the typical $\scrO(\log(N)/\sqrt{N})$ complexity estimate for proximal gradient methods. Indeed, such a step size choice yields the iteration complexity upper bound
\[
\Ex\left[\norm{\tilde{\scrG}_{\eta,\alpha_{\kappa}}(y_{\kappa})}^{2}\right]\leq \frac{\frac{\lip_{1}(h_{\eta})}{\beta}(\Psi_{\eta}(y_{1})-\Psi_{\eta}^{\Opt})+\frac{2\cs^{2}}{\ca}(1+\log(N))}{\sqrt{N}}.
\]
On the contrary, if a constant step size and constant mini-batch estimation strategy is adopted, then we see that convergence with respect to our stationary measure can only happen up to a plateau, a well-known fact when using stochastic approximation \cite{BotCurNoc18,GhaLanHon16}. Specifically, taking constant mini-batches $m_{k+1}=m$ and constant step-sizes $\alpha_{k}=\frac{2\beta}{\lip_{1}(h_{\eta})}$ for all $k\in\{1,\ldots,N\}$ and some $\beta\in (0,1/2)$, then our complexity bound is readily seen to become 
\[
\Ex\left[\norm{\tilde{\scrG}_{\eta,\alpha_{\kappa}}(y_{\kappa})}^{2}\right]\leq \frac{\lip_{1}(h_{\eta})(\Psi_{\eta}(y_{1})-\Psi_{\eta}^{\Opt})}{\beta(2-\beta)N}+\frac{2\cs^{2}}{(2-\beta)m}.
\]
Our next result measures is a complexity estimate in terms of the prox-gradient mapping involving the deterministic gradient $\nabla h_{\eta}$, instead of the stochastic approximation. 
\begin{corollary}\label{corr:Main}
Let $(y_k)_{k\in\N}$ be generated by Algorithm~\ref{alg:proxderivativefree} with gradient estimator \eqref{eq:V}. Assume that the step sizes $\alpha_{k}$ are chosen such that $\alpha_{k}\in(0,2/\lip_{1}(h_{\eta})]$, with $\alpha_{k}<2/\lip_{1}(h_{\eta})$ for at least one $k\in\{1,\ldots,N\}$. Let $\kappa:\Omega\to\{1,\ldots,N\}$ the discrete random variable with distribution \eqref{eq:lawkappa}. Then, 
\begin{equation}\label{eq:GNonconvex}
\Ex[\norm{\scrG_{\eta,\alpha_{\kappa}}(y_{\kappa})}^{2}]\leq \frac{4(\Psi_{\eta}(y_{1})-\Psi_{\eta}^{\Opt})}{\sum_{t=1}^{N}(2\alpha_{t}-\alpha^{2}_{t}\lip_{1}(h_{\eta}))}+\frac{\sum_{k=1}^{N}\frac{2\cs^{2}}{m_{k+1}}(4\alpha_{k}-\alpha^{2}_{k}\lip_{1}(h_{\eta}))}{\sum_{t=1}^{N}(2\alpha_{t}-\alpha^{2}_{t}\lip_{1}(h_{\eta}))}
\end{equation}
\end{corollary}
\begin{proof} 
From Lemma \ref{lem:RelationGradient} we readily obtain 
\[
\frac{\alpha_{k}}{2}(1-\frac{\alpha_{k}\lip_{1}(h_{\eta})}{2})\norm{\scrG_{\eta,\alpha_{k}}(y_{k})}^{2}\leq \alpha_{k}(1-\frac{\alpha_{k}\lip_{1}(h_{\eta})}{2})\norm{\tilde{\scrG}_{\eta,\alpha_{k}}(y_{k})}^{2}+\alpha_{k}(1-\frac{\alpha_{k}\lip_{1}(h_{\eta})}{2})\norm{\Delta W_{k+1}}^{2}_{\ast}. 
\]
Consequently, using \eqref{eq:complexityGap}:
\begin{align*}
\frac{1}{2}\Ex[\norm{\scrG_{\eta,\alpha_{\kappa}}(y_{\kappa})}^{2}]&=\frac{1}{2}\sum_{k=1}^{N}\frac{\alpha_{k}-\alpha^{2}_{k}\lip_{1}(h_{\eta})/2}{\sum_{t=1}^{N}(\alpha_{t}-\alpha^{2}_{t}\lip_{1}(h_{\eta})/2)}\Ex[\norm{\scrG_{\eta,\alpha_{k}}(y_{k})}^{2}]\\
&\leq\sum_{k=1}^{N}\frac{\alpha_{k}-\alpha_{k}^{2}\lip_{1}(h_{\eta})/2}{\sum_{t=1}^{N}(\alpha_{t}-\alpha^{2}_{t}\lip_{1}(h_{\eta})/2)}\Ex[\norm{\tilde{\scrG}_{\eta,\alpha_{k}}(y_{k})}^{2}]+\sum_{k=1}^{N}\frac{\alpha_{k}(1-\frac{\alpha_{k}\lip_{1}(h_{\eta})}{2})\Ex[\norm{\Delta W_{k+1}}^{2}_{\ast}]}{\sum_{t=1}^{N}(\alpha_{t}-\alpha^{2}_{t}\lip_{1}(h_{\eta})/2)}\\
&\leq \frac{\Psi_{\eta}(y_{1})-\Psi_{\eta}^{\Opt}+\sum_{k=1}^{N}\frac{\alpha_{k}\cs^{2}}{m_{k+1}}}{\sum_{t=1}^{N}(\alpha_{t}-\alpha^{2}_{t}\lip_{1}(h_{\eta})/2)}+\sum_{k=1}^{N}\frac{\alpha_{k}(1-\frac{\alpha_{k}\lip_{1}(h_{\eta})}{2})\frac{\cs^{2}}{m_{k+1}}}{\sum_{t=1}^{N}(\alpha_{t}-\alpha^{2}_{t}\lip_{1}(h_{\eta})/2)}\\
&=\frac{\Psi_{\eta}(y_{1})-\Psi^{\Opt}_{\eta}}{\sum_{t=1}^{N}(\alpha_{t}-\alpha^{2}_{t}\lip_{1}(h_{\eta})/2)}+\frac{\sum_{k=1}^{N}\alpha_{k}\frac{\cs^{2}}{m_{k+1}}(2-\frac{\alpha_{k}\lip_{1}(h_{\eta})}{2})}{\sum_{t=1}^{N}(\alpha_{t}-\alpha^{2}_{t}\lip_{1}(h_{\eta})/2)}
\end{align*}
\end{proof}
\begin{corollary}\label{corr:Main2} 
Let $(y_{k})_{k=0}^{N}$ be generated by Algorithm \ref{alg:proxderivativefree} with gradient estimator \eqref{eq:V}. Choosing the step size $\alpha_{k}=\frac{2\beta}{\lip_{1}(h_{\eta})\sqrt{k}},\beta\in(0,1/2)$, and the sampling rate $m_{k+1}=\ca\sqrt{k},\ca>0$, and the time window $N\geq 2$. Then, we have 
\begin{align*}
\Ex[\norm{\scrG_{\eta,\alpha_{\kappa}}(y_{\kappa})}^{2}]&\leq \frac{2 \lip_{1}(h_{\eta})(\Psi_{\eta}(y_{1})-\Psi_{\eta}^{\Opt})}{\beta\sqrt{N}}+\frac{\frac{8\cs^{2}}{\ca}(1+\log(N))}{\sqrt{N}}.
\end{align*}
The total number of calls to the stochastic oracle and lower level solutions to find a point $y\in\scrY$ such that $\Ex[\norm{\scrG_{\eta}(y)}^{2}]\leq \eps$ is bounded by $\scrO(\eps^{-3})$.
\end{corollary}
\begin{proof}
We start with recalling a simple integral bound. Note that 
$$
\sum_{t=1}^{N}\frac{1}{\sqrt{t}}\geq \int_{0}^{N}\frac{1}{\sqrt{x+1}}\dif x=2\sqrt{N+1}-2\geq\sqrt{N}
$$
for $N\geq 2$. Using this bound, the specific choices for the step sizes and the mini-batch size, lead to the following inequalities:
\begin{align*}
\Ex[\norm{\scrG_{\eta,\alpha_{\kappa}}(y_{\kappa})}^{2}]&\leq \frac{4(\Psi_{\eta}(y_{1})-\Psi_{\eta}^{\Opt})}{\sum_{t=1}^{N}\alpha_{t}}+\frac{\sum_{k=1}^{N}\frac{2\cs^{2}}{m_{k+1}}(4\alpha_{k}-\alpha^{2}_{k}\lip_{1}(h_{\eta}))}{\sum_{t=1}^{N}\alpha_{t}}\\
&\leq \frac{4(\Psi_{\eta}(y_{1})-\Psi_{\eta}^{\Opt})}{\sum_{t=1}^{N}\alpha_{t}}+\frac{\sum_{k=1}^{N}\frac{8\cs^{2}}{m_{k+1}}\alpha_{k}}{\sum_{t=1}^{N}\alpha_{t}}\\
&\leq\frac{2\lip_{1}(h_{\eta})(\Psi_{\eta}(y_{1})-\Psi_{\eta}^{\Opt})}{\beta\sqrt{N}}+\frac{\frac{8\cs^{2}}{\ca}(1+\log(N))}{\sqrt{N}}.
\end{align*}
Hence, the iteration complexity of the method is bounded by $\scrO(\eps^{-2})$. Now, to bound the oracle complexity, note that in each iteration of Algorithm \ref{alg:proxderivativefree} we need $m_{k+1}$ Gaussian vectors $U$ and the same number of random vectors $\xi=(\xi_{1},\xi_{2})$ to construct the random vector $\sum_{i=1}^{m_{k+1}}\frac{H(y^{k}+\eta U^{i}_{k+1},\xi^{i}_{k+1})}{\eta} BU^{i}_{k+1}$. We therefore have $m_{k+1}$ calls of the stochastic function $H(\cdot,\xi)$ in every single iteration. The total number of calls is thus 
$\sum_{k=1}^{N}m_{k+1}=\ca\sum_{k=1}^{N}\sqrt{k}\leq \frac{2\ca}{3}N^{3/2}$
As $N=\scrO(\eps^{-2})$, the oracle complexity is upper bounded by $\scrO(\eps^{-3})$. Similarly, in every iteration we need $m_{k+1}$ solutions of the lower level problem. Hence, by the above computation, the total number of lower level solves is bounded by $\scrO(\eps^{-3})$. 
\end{proof}

\subsection{Inexact lower level solution}
\label{sec:inexact}
%
For the complexity analysis of the inexact regime, we have to adapt the definition of the gradient mapping accordingly to 
\begin{equation}
\tilde{\scrG}^{\beta}_{\eta,t}(\by,\omega)\eqdef \frac{1}{t}\left(\by-P_{t}(\by,\hat{V}^{\beta}_{\eta,m}(\by,\omega))\right).
\end{equation}
Using this merit function and the definition of the error increment 
\begin{align*}
\Delta W^{\beta}_{k+1}&\eqdef \hat{V}^{\beta_{k}}_{k+1}-\nabla h_{\eta}(y_{k})=\hat{V}_{k+1}-a_{k+1}+b_{k+1}-\nabla h_{\eta}(y_{k})=\Delta W_{k+1}-a_{k+1}+b_{k+1},
\end{align*}
we can repeat the one-step analysis of the exact case to obtain the bound 
\begin{equation}\label{eq:iterative_error_decay}
\begin{split}
\Psi_{\eta}(y_{k+1})-\Psi_{\eta}(y_{k})\leq& -\alpha_{k}\left(1-\frac{\alpha_{k}\lip_{1}(h_{\eta})}{2}\right)\norm{\tilde{\scrG}^{\beta_{k}}_{\eta,\alpha_{k}}(y_{k})}^{2}\\
&+\alpha_{k}\inner{\Delta W^{\beta}_{k+1},\scrG_{\eta,\alpha_{k}}(y_{k})}+\alpha_{k}\norm{\Delta W^{\beta}_{k+1}}^{2}_{\ast}
\end{split}
\end{equation}
Lemma~\ref{lem:RelationGradient} generalizes in the inexact case in the following way:
\begin{lemma}\label{lem:RelationBiasedGradient}
We have 
\begin{equation}
\norm{\scrG_{\eta,t}(\by)}^{2}\leq 2\norm{\tilde{\scrG}_{\eta,t}^\beta(\by)}^{2}+2\norm{\Delta W_{\eta,m}^\beta(\by)}^{2}_{\ast}\qquad \text{a.s.} 
\end{equation}
\end{lemma}
\begin{proof}
The assertion follows line by line as in Lemma~\ref{lem:RelationGradient} by replacing $V_{\eta,m}(\by)$ with $V_{\eta,m}^\beta(\by)$.
\end{proof}

Using this lemma directly in the penultimate display, we see that for $\alpha_{k}\in(0,2/\lip_{1}(h_{\eta})]$ 
\begin{align*}
\Psi_{\eta}(y_{k+1})-\Psi_{\eta}(y_{k})&\leq -\frac{\alpha_{k}}{2}\left(1-\frac{\alpha_{k}\lip_{1}(h_{\eta})}{2}\right)\norm{\scrG_{\eta,\alpha_{k}}(y_{k})}^{2}+\alpha_{k}\inner{\Delta W^{\beta}_{k+1},\scrG_{\eta,\alpha_{k}}(y_{k})}+\alpha_{k}\norm{\Delta W^{\beta}_{k+1}}^{2}_{\ast}\\
&+\alpha_{k}\left(1-\frac{\alpha_{k}\lip_{1}(h_{\eta})}{2}\right)\norm{\Delta W^{\beta}_{k+1}}^{2}_{\ast}.
\end{align*}
Applying the Young's inequality of the inner product, we conclude that for arbitrary $\delta>0$ 
\begin{align*}
\Psi_{\eta}(y_{k+1})-\Psi_{\eta}(y_{k})\leq -\frac{\alpha_{k}}{2}\left(1-\frac{1}{\delta}-\frac{\alpha_{k}\lip_{1}(h_{\eta})}{2}\right)\norm{\scrG_{\eta,\alpha_{k}}(y_{k})}^{2}+\alpha_{k}\left(2+\frac{\delta}{2}-\frac{\alpha_{k}\lip_{1}(h_{\eta})}{2}\right)\norm{\Delta W^{\beta}_{k+1}}^{2}_{\ast}.
\end{align*}
Rearranging this expression and summing both sides from $k=1$ to $N$, we remain with 
\[\sum_{k=1}^{N}\frac{\alpha_{k}}{2}\left(\frac{\delta-1}{\delta}-\frac{\alpha_{k}\lip_{1}(h_{\eta})}{2}\right)\norm{\scrG_{\eta,\alpha_{k}}(y_{k})}^{2}\leq \Psi_{\eta}(y_{1})-\Psi^{\Opt}_{\eta}
+\sum_{k=1}^{N}\alpha_{k}\left(\frac{4+\delta}{2}-\frac{\alpha_{k}\lip_{1}(h_{\eta})}{2}\right)\norm{\Delta W^{\beta}_{k+1}}^{2}_{\ast}.\]
Since $\norm{\Delta W^{\beta}_{k+1}}^{2}_{\ast}\leq 3\norm{\Delta W_{k+1}}^{2}_{\ast}+3\norm{a_{k+1}}^{2}_{\ast}+3\norm{b_{k+1}}^{2}_{\ast}$, we can take iteratively conditional expectations to obtain the main complexity bound for the inexact regime. 
\begin{theorem}\label{th:complexityInexact}
Let $(y_{k})_{k=0}^{N}$ be generated by Algorithm \ref{alg:proxderivativefree} with inexact gradient estimator \eqref{eq:Vinexact}, $\delta>1$ and $r,s\geq 1$ such that $\frac{2s(r-1)}{r}=p\geq 2$, where $p$ is the exponent in Definition \ref{def:inexact_lower}. Suppose that the step sizes $\alpha_{k}$ are chosen such that $\alpha_{k}\in(0,\frac{2(\delta-1)}{\delta\lip_{1}(h_{\eta})}]$, with $\alpha_{k}<\frac{2(\delta-1)}{\delta\lip_{1}(h_{\eta})}$ for at least on $k\in\{1,\ldots,N\}$. On $(\Omega,\scrF,\Pr)$ define an independent random variable $\kappa:\Omega\to\{1,\ldots,N\}$ with probability mass function 
$$
p(k)=\Pr(\kappa=k)\eqdef\frac{\alpha_{k}\frac{\delta-1}{\delta}-\alpha_{k}^{2}\lip_{1}(h_{\eta})/2}{\sum_{t=1}^{N}(\alpha_{t}\frac{\delta-1}{\delta}-\alpha^{2}_{t}\lip_{1}(h_{\eta})/2)}\quad \forall k\in\{1,\ldots,N\}.
$$
Let $D_{k}\eqdef \frac{3(4+\delta)}{2}\left(\frac{\cs^{2}}{m_{k+1}}+C_{F}\frac{2\beta^{2}_{k}}{\eta^{2}}\right)$. Then, 
\begin{equation}
\frac{1}{2}\Ex[\norm{\scrG_{\eta,\alpha_{\kappa}}(y_{\kappa})}^{2}]\leq \frac{\Psi_{\eta}(y_{1})-\Psi^{\Opt}_{\eta}}{\sum_{t=1}^{N}(\alpha_{t}\frac{\delta-1}{\delta}-\alpha^{2}_{t}\lip_{1}(h_{\eta})/2)}+\frac{\sum_{k=1}^{N}\alpha_{k}D_{k}}{\sum_{t=1}^{N}(\alpha_{t}\frac{\delta-1}{\delta}-\alpha^{2}_{t}\lip_{1}(h_{\eta})/2)}.
\end{equation}
\end{theorem}

Similarly as in the exact case, we again find a trade-off between aggressive step-sizes and the size of the mini-batches. However, in the inexact computational model, we additionally observe a trade-off between the step-size schedule and the accuracy tolerance $\beta_k$ in the lower level problem. 

\begin{corollary}
Let be $\delta>1$ and consider a step-size $\alpha_k\le \frac{\min\{\delta-1,1\}}{\lip_1(h_\eta)}$. Then
\[
\frac{1}{2}\Ex[\norm{\scrG_{\eta,\alpha_{\kappa}}(y_{\kappa})}^{2}]\leq \frac{\Psi_{\eta}(y_{1})-\Psi^{\Opt}_{\eta}}{\sum_{t=1}^{N}(\frac{\alpha_{t}}{2})}+\frac{\sum_{k=1}^{N}\alpha_{k}D (\frac{1}{m_k}+\beta_k^2)}{\sum_{t=1}^{N}(\frac{\alpha_{t}}{2})},
\]
where $D\eqdef 3(4+\delta)\max\{\cs^{2},\frac{2}{\eta^{2}}C_{F}\}$. 
\end{corollary}
The constants appearing in the upper complexity bound can be well balanced via a judicious choice of $\delta$. For instance, setting $\delta = 2$, the step-size policy $\alpha_k = \frac{2\beta}{\lip_1(h_\eta)\sqrt{k}}$ with $\beta\in(0,1/2)$, and choosing the sampling rate $m_k = a\sqrt{k}$ and the accuracy tolerance $\beta_k = \cb k^{-\frac14}$ with constants $\ca,\cb>0$, we obtain the overall complexity estimate
\[
\frac{1}{2}\Ex[\norm{\scrG_{\eta,\alpha_{\kappa}}(y_{\kappa})}^{2}]\leq \frac{\frac{\lip_1(h_\eta)}{\beta} (\Psi_{\eta}(y_{1})-\Psi^{\Opt}_{\eta}) + 2D(\frac1a+\cb^2)(1+\log(N))}{\sqrt{N}},
\]
which resembles those of the exact oracle case.

\section{The convex case with inexact lower level solution}
\label{sec:convex}
%

We now turn to the case in which the implicit function $h$ is convex. In this special setting, the smoothed function $h_{\eta}$ is also convex and Lipschitz continuous. By the subgradient inequality, we have for all $\by\in\scrY$ and $g\in\partial h(\by)$ 
\begin{equation}\label{eq:boundh_Convex}
h_{\eta}(\by)=\Ex[h(\by+\eta U)]\geq\Ex[h(\by)+\inner{g,\eta U}]=h(\by).
\end{equation}
Moreover, in the convex case, it holds true that $\nabla h_{\eta}(\by)$ always belongs to some $\delta$-subdifferential of the function $h$ (cf. Lemma \ref{lem:deltagradient}). In this section, we make an additional boundedness assumption on the bilevel problem. 
\begin{assumption}\label{ass:dombounded}
The domain $\dom(r_{1})$ is bounded.
\end{assumption}
This assumption is natural in many inverse problem settings where function $r_{1}$ takes over the role of a hard penalty enforcing constraints on the learned parameters, or imposes a-priori structure like (group)-sparsity. 

\begin{theorem}\label{th:convex}
Let $(y_{k})_{k=1}^{N}$ be generated by Algorithm \ref{alg:proxderivativefree} gradient estimator \eqref{eq:Vinexact}. Assume that the implicit function $y\mapsto h(y)$ is convex and Assumption \ref{ass:dombounded} holds. Consider a step size policy $(\alpha_{k})_{k=1}^{N}$ satisfying 
\begin{equation}\label{eq:Stepconvex}
0<\alpha_{N}\leq\alpha_{N-1}\leq\ldots\leq\alpha_{1}\leq\frac{1}{\lip_{1}(h_{\eta})}\text{ and } \alpha_{k}+\alpha_{k-1}\leq\frac{1}{\lip_{1}(h_{\eta})},\quad \text{for all}\ k=2,\dots,N.
\end{equation}
Let $\kappa:\Omega\to \{1,\ldots,N\}$ be an independent random variable, with probability mass function
\begin{equation}
p(k)=\Pr(\kappa=k)\eqdef \frac{\ca_{k}}{A_{N}},\quad A_{N}\eqdef\sum_{t=1}^{N}\ca_{t},\;\ca_{k}\eqdef \alpha_{k}-\alpha^{2}_{k}\lip_{1}(h_{\eta}).
\end{equation}
Then, we have
\begin{equation}
\Ex[\Psi(y_{\kappa})-\Psi^{\rm Opt}]\leq \frac{\sum_{k=1}^{N}\frac{\alpha^{2}_{k}}{m_{k+1}}D_{k}+\frac{M\sqrt{C_{F}}}{\eta}\sum_{k=1}^{N}\ca_{k}\beta_{k}+M/2+\alpha_{1}\Delta\Psi_{1}}{A_{N}}+\eta\sqrt{n}\lip_{0}(h),
\end{equation}
where $D_{k}\eqdef 3\left(\cs^{2}/2+\frac{\beta_{k}^{2}m_{k+1}}{\eta^{2}}C_{F}\right)$ and $M\eqdef\sup_{y_{1},y_{2}\in\dom(r_{1})}\norm{y_{1}-y_{2}}^{2}$.
\end{theorem}
\begin{proof}
Let $\by^{\ast}$ denote a solution of the original problem \eqref{eq:ImplicitProg}. Let $(\alpha_{k})_{k}$ be a sequence of step-sizes, satisfying $0\leq\alpha_{k}<\frac{1}{\lip_{1}(h_{\eta})}.$ For $\eta>0$ we then have 
\begin{align*}
\Psi_{\eta}(y_{k+1})-\Psi_{\eta}(\by^{\ast})&=h_{\eta}(y_{k+1})-h_{\eta}(y_{k})+h_{\eta}(y_{k})-h_{\eta}(\by^{\ast})+r_{1}(y_{k+1})-r_{1}(y_{k}). 
\end{align*}
Using the descent property \eqref{eq:descent} and the convexity of the smoothed implicit function $h_{\eta}$, we deduce that 
\begin{align*}
&h_{\eta}(y_{k+1})-h_{\eta}(y_{k})\leq \inner{\nabla h_{\eta}(y_{k}),y_{k+1}-y_{k}}+\frac{\lip_{1}(h_{\eta})}{2}\norm{y_{k+1}-y_{k}}^{2},\text{ and }\\ 
&h_{\eta}(y_{k})-h_{\eta}(\by^{\ast})\leq \inner{\nabla h_{\eta}(y_{k}),y_{k}-\by^{\ast}}. 
\end{align*}
Recall that $\Delta W^{\beta}_{k+1}=\hat{V}^{\beta}_{k+1}-\nabla h_{\eta}(y_{k})$. Then, we continue from the above with 
\begin{align*}
\Psi_{\eta}(y_{k+1})-\Psi_{\eta}(\by^{\ast})&\leq \inner{\Delta W^{\beta}_{k+1},y_{k}-y_{k+1}}+\frac{\lip_{1}(h_{\eta})}{2}\norm{y_{k+1}-y_{k}}^{2}\\
&+\inner{\Delta W^{\beta}_{k+1},\by^{\ast}-y_{k}}+r_{1}(y_{k+1})-r_{1}(\by^{\ast})+\inner{\hat{V}^{\beta}_{k+1},y_{k+1}-\by^{\ast}}.
\end{align*}
By definition of the point $y_{k+1}$, we have 
\[
r_{1}(\by^{\ast})\geq r_{1}(y_{k+1})+\frac{1}{\alpha_{k}}\inner{B(y_{k}-y_{k+1}),\by^{\ast}-y_{k+1}}-\inner{\hat{V}^{\beta}_{k+1},\by^{\ast}-y_{k+1}}. 
\]
Combining these two estimates, we can continue with
\begin{align*}
\Psi_{\eta}(y_{k+1})-\Psi_{\eta}(\by^{\ast})&\leq \inner{\Delta W^{\beta}_{k+1},y_{k}-y_{k+1}}+\frac{\lip_{1}(h_{\eta})}{2}\norm{y_{k+1}-y_{k}}^{2}+\inner{\Delta W^{\beta}_{k+1},y_{k}-\by^{\ast}}\\
&+\frac{1}{\alpha_{k}}\inner{B(y_{k}-y_{k+1}),y_{k+1}-\by^{\ast}}\\
&=\inner{\Delta W^{\beta}_{k+1},y_{k}-y_{k+1}}+\frac{\lip_{1}(h_{\eta})}{2}\norm{y_{k+1}-y_{k}}^{2}+\inner{\Delta W^{\beta}_{k+1},\by^{\ast}-y_{k}}\\
&+\frac{1}{\alpha_{k}}\left[\frac{1}{2}\norm{y_{k}-\by^{\ast}}^{2}-\frac{1}{2}\norm{y_{k+1}-y_{k}}^{2}-\frac{1}{2}\norm{y_{k+1}-\by^{\ast}}^{2}\right].
\end{align*}
Note that $ax-\frac{bx^{2}}{2}\leq\frac{a^{2}}{2b}$ for all $x\geq 0$, implying that 
\begin{align*}
\inner{\Delta W^{\beta}_{k+1},y_{k}-y_{k+1}}&+\frac{\lip_{1}(h_{\eta})\alpha_{k}-1}{2\alpha_{k}}\norm{y_{k+1}-y_{k}}^{2}\leq \norm{\Delta W^{\beta}_{k+1}}_{\ast}\cdot\norm{y_{k+1}-y_{k}}+\frac{\lip_{1}(h_{\eta})\alpha_{k}-1}{2\alpha_{k}}\norm{y_{k+1}-y_{k}}^{2}\\
&\leq \frac{\alpha_{k}}{2(1-\alpha_{k}\lip_{1}(h_{\eta}))}\norm{\Delta W^{\beta}_{k+1}}_{\ast}^{2}.
\end{align*}
Thus, multiplying both sides in the penultimate display by $(\alpha_{k}-\alpha^{2}_{k}\lip_{1}(h_{\eta}))$, we can continue the bound by 
\begin{align*}
(\alpha_{k}-\alpha^{2}_{k}\lip_{1}(h_{\eta}))[\Psi_{\eta}(y_{k+1})-\Psi_{\eta}(\by^{\ast})]&\leq\frac{\alpha^{2}_{k}}{2}\norm{\Delta W^{\beta}_{k+1}}^{2}_{\ast}+(\alpha_{k}-\alpha_{k}^{2}\lip_{1}(h_{\eta}))\inner{\Delta W^{\beta}_{k+1},\by^{\ast}-y_{k}}\\
&\quad+(1-\alpha_{k}\lip_{1}(h_{\eta}))\left[\frac{1}{2}\norm{y_{k}-\by^{\ast}}^{2}-\frac{1}{2}\norm{y_{k+1}-\by^{\ast}}^{2}\right].
\end{align*}
Using \eqref{eq:boundh_Convex}, we note that $\Psi_{\eta}(y_{k+1})\geq \Psi(y_{k+1})$. Additionally, Lemma \ref{lem:boundfunctions} yields $\Psi_{\eta}(\by^{\ast})\geq\Psi^{\rm Opt}-\eta\lip_{0}(h)\sqrt{n}$. This allows us to bound the objective function gap by 
\begin{align*}
(\alpha_{k}-\alpha^{2}_{k}\lip_{1}(h_{\eta}))[\Psi(y_{k+1})-\Psi^{\rm Opt}]&\leq (\alpha_{k}-\alpha^{2}_{k}\lip_{1}(h_{\eta}))[\Psi_{\eta}(y_{k+1})-\Psi_{\eta}(\by^{\ast})]+\eta\lip_{0}(h)\sqrt{n}(\alpha_{k}-\alpha^{2}_{k}\lip_{1}(h_{\eta}))\\ 
&\leq \frac{\alpha^{2}_{k}}{2}\norm{\Delta W^{\beta}_{k+1}}^{2}_{\ast}+(\alpha_{k}-\alpha_{k}^{2}\lip_{1}(h_{\eta}))\inner{\Delta W^{\beta}_{k+1},\by^{\ast}-y_{k}}\\
&\quad+(1-\alpha_{k}\lip_{1}(h_{\eta}))\left[\frac{1}{2}\norm{y_{k}-\by^{\ast}}^{2}-\frac{1}{2}\norm{y_{k+1}-\by^{\ast}}^{2}\right]\\
&\quad+\eta\lip_{0}(h)\sqrt{n}(\alpha_{k}-\alpha^{2}_{k}\lip_{1}(h_{\eta})).
\end{align*}
To bound the terms on the right-hand side, we first use error decomposition \eqref{eq:errordecomp} to bound the first addendum by 
$\norm{\Delta W^{\beta}_{k+1}}^{2}_{\ast}\leq 3\norm{\Delta W_{k+1}}^{2}_{\ast}+3\norm{a_{k+1}}^{2}_{\ast}+3\norm{b_{k+1}}^{2}_{\ast}$, 
as well as the second addendum $\inner{\Delta W^{\beta}_{k+1},\by^{\ast}-y_{k}}=\inner{\Delta W_{k+1},\by^{\ast}-y_{k}}-\inner{a_{k+1},\by^{\ast}-y_{k}}+\inner{b_{k+1},\by^{\ast}-y_{k}}.$ Hence, taking conditional expectations on both sides, we continue with 
\begin{align*}
\Ex_{k}&\left[(\alpha_{k}-\alpha^{2}_{k}\lip_{1}(h_{\eta}))(\Psi(y_{k+1})-\Psi(\by^{\ast}))\right]\leq \frac{3\alpha^{2}_{k}}{2}\frac{\cs^{2}}{m_{k+1}}+3\alpha_{k}^{2}\frac{n\beta_{k}^{2}}{\eta^{2}}C_{F}+\Ex_{k}\left[(\alpha_{k}-\alpha_{k}^{2}\lip_{1}(h_{\eta}))\inner{b_{k+1},y_{k}-\by^{\ast}}\right]\\
&+(1-\alpha_{k}\lip_{1}(h_{\eta}))\Ex_{k}\left[\frac{1}{2}\norm{y_{k}-\by^{\ast}}^{2}-\frac{1}{2}\norm{y_{k+1}-\by^{\ast}}^{2}\right]+\eta\lip_{0}(h)\sqrt{n}(\alpha_{k}-\alpha^{2}_{k}\lip_{1}(h_{\eta}))\\
&\leq \frac{3\alpha^{2}_{k}}{2}\frac{\cs^{2}}{m_{k+1}}+3\alpha_{k}^{2}\frac{n\beta_{k}^{2}}{\eta^{2}}C_{F}+M(\alpha_{k}-\alpha_{k}^{2}\lip_{1}(h_{\eta})\Ex_{k}\left[\norm{b_{k+1}}_{\ast}\right]\\
&+(1-\alpha_{k}\lip_{1}(h_{\eta}))\Ex_{k}\left[\frac{1}{2}\norm{y_{k}-\by^{\ast}}^{2}-\frac{1}{2}\norm{y_{k+1}-\by^{\ast}}^{2}\right]+\eta\lip_{0}(h)\sqrt{n}(\alpha_{k}-\alpha^{2}_{k}\lip_{1}(h_{\eta})),
\end{align*}
where the second inequality uses Cauchy-Schwarz and the bound $M\geq\norm{y_{k}-\by^{\ast}}^{2}$, which holds thanks to Assumption \ref{ass:dombounded}. Since the step size sequence $(\alpha_{k})_{k}$ is non-decreasing and satisfies condition \eqref{eq:Stepconvex}, we can continue to obtain 
\begin{align*}
\sum_{k=1}^{N} & (1-\alpha_{k}\lip_{1}(h_{\eta}))\left(\frac{1}{2}\norm{y_{k}-\by^{\ast}}^{2}-\frac{1}{2}\norm{y_{k+1}-\by^{\ast}}^{2}\right)\\
&=(1-\alpha_{1}\lip_{1}(h_{\eta}))\frac{1}{2}\norm{y_{1}-\by^{\ast}}^{2}+\sum_{k=2}^{N}\lip_{1}(h_{\eta})(\alpha_{k}-\alpha_{k+1})\frac{1}{2}\norm{y_{k+1}-\by^{\ast}}^{2}-(1-\lip_{1}(h_{\eta})\alpha_{N})\frac{1}{2}\norm{y_{N+1}-\by^{\ast}}^{2}\\
&\leq (1-\alpha_{N}\lip_{1}(h_{\eta}))\frac{M}{2}.
\end{align*}
Next, calling $\Delta\Psi_{k}\eqdef \Psi(y_{k})-\Psi^{\rm Opt}$ and $\ca_{k}\eqdef\alpha_{k}-\alpha_{k}^{2}\lip_{1}(h_{\eta})$, we deduce that 
\begin{align*}
 \sum_{k=1}^{N}\ca_{k}\Delta\Psi_{k}&=\sum_{k=1}^{N}\ca_{k}\Delta\Psi_{k+1}+\sum_{k=1}^{N}\ca_{k}(\Delta\Psi_{k}-\Delta\Psi_{k+1}), \text{ and} \\
 \sum_{k=1}^{N}\ca_{k}(\Delta\Psi_{k}-\Delta\Psi_{k+1})&=\sum_{k=1}^{N}\ca_{k}\Delta\Psi_{k}-\sum_{k=1}^{N}\ca_{k}\Delta\Psi_{k+1}\\
 &=\ca_{1}\Delta\Psi_{1}+\sum_{k=2}^{N}\ca_{k}\Delta\Psi_{k}-\sum_{k=1}^{N}\ca_{k}\Delta\Psi_{k+1}\\
 &\leq \ca_{1}\Delta\Psi_{1}+\sum_{k=2}^{N}\ca_{k-1}\Delta\Psi_{k}-\sum_{k=1}^{N}\ca_{k}\Delta\Psi_{k+1}\\
 &\leq \ca_{1}\Delta\Psi_{1}.
 \end{align*}
The third inequality uses the relation $\ca_{k}\leq\ca_{k-1}$.\footnote{This can be deduced as follows: Since $\alpha_{k}\leq\alpha_{k-1}$ one easily sees that
\begin{align*}
\ca_{k}-\ca_{k-1}&=(\alpha_{k}-\alpha_{k-1})-\lip_{1}(h_{\eta})(\alpha_{k}^{2}-\alpha_{k-1}^{2})=(\alpha_{k}-\alpha_{k-1})\left(1-\lip_{1}(h_{\eta})(\alpha_{k}+\alpha_{k-1})\right)\leq 0.
\end{align*}
}
Taking full expectations and summing from $k=1,\ldots,N$, we continue the above bound 
\begin{align*}
\Ex&\left[\sum_{k=1}^{N}(\alpha_{k}-\alpha^{2}_{k}\lip_{1}(h_{\eta}))\Delta\Psi_{k}\right]
\leq\Ex\left[\sum_{k=1}^{N}(\alpha_{k}-\alpha^{2}_{k}\lip_{1}(h_{\eta})\Delta\Psi_{k+1}\right]+(\alpha_{1}-\lip_{1}(h_{\eta})\alpha_{1}^{2})\Delta\Psi_{1}\\
&\leq \sum_{k=1}^{N}\left(\frac{3\alpha^{2}_{k}}{2}\frac{\cs^{2}}{m_{k+1}}+\frac{3\alpha_{k}^{2}\beta_{k}^{2}}{\eta^{2}} C_{F}\right)+(1-\alpha_{N}\lip_{1}(h_{\eta}))\frac{M}{2}\\
&+\eta\lip_{0}(h)\sqrt{n}\sum_{k=1}^{N}(\alpha_{k}-\alpha^{2}_{k}\lip_{1}(h_{\eta}))+M\sum_{k=1}^{N}(\alpha_{k}-\alpha^{2}_{k}\lip_{1}(h_{\eta}))\Ex\left(\norm{b_{k+1}}_{\ast}\right)\\
&+(\alpha_{1}-\alpha^{2}_{1}\lip_{1}(h_{\eta}))\Delta\Psi_{1}\\
&\leq \sum_{k=1}^{N}\left(\frac{3\alpha^{2}_{k}}{2}\frac{\cs^{2}}{m_{k+1}}+\frac{3\alpha_{k}^{2}\beta_{k}^{2}}{\eta^{2}}C_{F}\right)+M\sum_{k=1}^{N}(\alpha_{k}-\alpha^{2}_{k}\lip_{1}(h_{\eta}))\frac{\beta_{k}}{\eta}\sqrt{C_{F}}+(1-\alpha_{N}\lip_{1}(h_{\eta}))\frac{M}{2}\\
&+\eta\lip_{0}(h)\sqrt{n}\sum_{k=1}^{N}(\alpha_{k}-\alpha^{2}_{k}\lip_{1}(h_{\eta}))+(\alpha_{1}-\alpha^{2}_{1}\lip_{1}(h_{\eta}))\Delta\Psi_{1}.
\end{align*}
Therefore, defining $D_{k}\eqdef 3\left(\cs^{2}/2+\frac{\beta_{k}^{2}m_{k+1}}{\eta^{2}}C_{F}\right)$, and constructing an independent random variable $\kappa:\Omega\to \{1,\ldots,N\}$ with density function 
$$
p(k)=\Pr(\kappa=k)=\frac{\ca_{k}}{\sum_{t=1}^{N}\ca_{t}}\equiv\frac{\ca_{k}}{A_{N}},\quad A_{N}\eqdef\sum_{t=1}^{N}\ca_{t},
$$
we obtain in a similar fashion as in the proof of Theorem \ref{th:nonconvex1}
\[
\Ex[\Delta\Psi_{\kappa}]\leq \frac{\sum_{k=1}^{N}\frac{\alpha^{2}_{k}}{m_{k+1}}D_{k}+\frac{M\sqrt{C_{F}}}{\eta}\sum_{k=1}^{N}\ca_{k}\beta_{k}+M/2+\alpha_{1}\Delta\Psi_{1}}{A_{N}}+\eta\sqrt{n}\lip_{0}(h).
\]
\end{proof}
Similar to the analysis in the non-convex case, we can simplify the complexity bound of Theorem \ref{th:convex} via a judicious selection of parameters. 
\begin{corollary}
Let be $\delta>1$ and consider step sizes $\alpha_{k}$ so that \eqref{eq:Stepconvex} holds true. Then 
\begin{equation}\label{eq:CorConvex1}
\Ex[\Delta\Psi_{\kappa}]\leq \frac{\sum_{k=1}^{N}\alpha^{2}_{k}(\frac{1}{m_{k+1}}+\beta^{2}_{k})\bar{D}+\frac{M\sqrt{C_{F}}}{\eta}\sum_{k=1}^{N}\alpha_{k}\beta_{k}+M/2+\alpha_{1}\Delta\Psi_{1}}{\sum_{k=1}^{N}\alpha_{k}/2}+\eta\sqrt{n}\lip_{0}(h).
\end{equation}
In particular, for fixed time horizon $N$, choosing step size $\alpha_{k}=\frac{\alpha_{0}}{\sqrt{N}}$, the constant mini-batch $m_{k+1}=m\geq 1$, and $\beta_{k}=\frac{a}{\sqrt{k+a}},a\geq 1$, as well as $\eta=\frac{1}{\sqrt{N}}$. we obtain
\begin{equation}\label{eq:CorConvex2}
\Ex[\Delta\Psi_{\kappa}]\leq \frac{\frac{\bar{D}\alpha_{0}}{m}+\frac{\bar{D}\alpha_{0}}{N}a^{2}\log(N+a)+\frac{M\sqrt{C_{F}}}{\eta}\frac{2\alpha_{0}a\sqrt{N+a}}{\sqrt{N}}+\frac{M}{2}+\alpha_{1}\Delta\Psi_{1}}{\frac{\sqrt{N}}{2}}+\frac{\sqrt{n}\lip_{0}(h)}{\sqrt{N}}.
\end{equation}
\end{corollary}
\begin{proof}
First we note that
$$
\sum_{k=1}^{N}\frac{\alpha^{2}_{k}}{m_{k+1}}D_{k}\leq \bar{D}\sum_{k=1}^{N}\alpha^{2}_{m}(\frac{1}{m_{k+1}}+\beta^{2}_{k}), 
$$
where $\bar{D}\eqdef3 \max\{\cs^{2}/2,\frac{C_{F}}{\eta^{2}}\}$; Second $\sum_{k=1}^{N}\ca_{k}\beta_{k}\leq\sum_{k=1}^{N}\alpha_{k}\beta_{k}$. Moreover, by choosing the step size $\alpha_{k}\leq\frac{1}{2\lip_{1}(h_{\eta})}$, we see that $\ca_{k}=\alpha_{k}-\alpha^{2}_{k}\lip_{1}(h_{\eta})\geq \frac{\alpha_{k}}{2}$. Combining all these estimates, we arrive at \eqref{eq:CorConvex1}. 
For fixed time horizon $N$, choose step size $\alpha_{k}=\frac{\alpha_{0}}{\sqrt{N}}$, the constant mini-batch $m_{k+1}=m\geq 1$, and $\beta_{k}=\frac{a}{\sqrt{k+a}},a\geq 1$, as well as $\eta=\frac{1}{\sqrt{N}}$. Substituting these numbers into expression \eqref{eq:CorConvex1}, we immediately obtain \eqref{eq:CorConvex2}.
\end{proof}

\section{Explicit complexity and relaxed stationarity}
\label{sec:Goldstein}
%

The previous results provided a finite-time complexity estimate in terms of the gradient mapping of the proximal gradient algorithm, involving the Gaussian smoothed objective. It is intuitive that a small proximal gradient in the smoothed regime should imply an approximate stationary point in the original optimization problem, when the smoothing parameter is sufficiently small. In this section we make this intuition precise and relate our complexity estimate from Theorem \ref{th:nonconvex1} to a complexity estimate with respect to a relaxed stationary point. 

Fix $\eta>0$ and define $\alpha_{1}=\frac{2\beta}{\lip_{1}(h_{\eta})}$. Define the auxiliary process $(\hat{y}_{k})_{k\geq 1}$ by 
\[
\hat{y}_{k}\eqdef P_{\alpha_{1}}(y_{k},\hat{V}_{k+1})=\argmin_{\bu}\{r_{1}(\bu)+\frac{1}{2\alpha_{1}}\norm{\bu-(y_{k}-\alpha_{1}B^{-1}\hat{V}_{k+1})}^{2}\}. 
\]
This point is uniquely characterized by the optimality condition 
\[
y_{k}-\alpha_{1}B^{-1}\hat{V}_{k+1}\in \hat{y}_{k}+\alpha_{1}B^{-1}\partial r_{1}(\hat{y}_{k}),
\]
or equivalently
\[
\hat{y}_{k}+\alpha_{1}B^{-1}\scrD(y_{k})\in\hat{y}_{k}+\alpha_{1}B^{-1}\partial r_{1}(\hat{y}_{k})\iff \scrD(y_{k})\in\partial r_{1}(\hat{y}_{k}),
\]
where 
\[
\scrD(y_{k})\eqdef B\left(\frac{y_{k}-\hat{y}_{k}}{\alpha_{1}}\right)-\nabla h_{\eta}(\hat{y}_{k})+(\nabla h_{\eta}(\hat{y}_{k})-\nabla h_{\eta}(y_{k}))+(\nabla h_{\eta}(y_{k})-\hat{V}_{k+1}). 
\]
This yields 
\[
B\left(\frac{y_{k}-\hat{y}_{k}}{\alpha_{1}}\right)+(\nabla h_{\eta}(\hat{y}_{k})-\nabla h_{\eta}(y_{k}))-\Delta W_{k+1}\in\partial r_{1}(\hat{y}_{k})+\nabla h_{\eta}(\hat{y}_{k}).
\]
From now on we continue our developments with Assumption \ref{ass:dombounded} in place. Choose $\eps_{1}>0,\eps_{2}>0$, and $\eta<\bar{\eta}$ (depending on $\eps_{1},\eps_{2}$), as defined in Proposition \ref{prop:epssmoothedgrad}, so that 
\[
B\left(\frac{y_{k}-\hat{y}_{k}}{\alpha_{1}}\right)+(\nabla h_{\eta}(\hat{y}_{k})-\nabla h_{\eta}(y_{k}))-\Delta W_{k+1}\in\partial r_{1}(\hat{y}_{k})+\partial_{G}^{\eps_{2}} h_{\eta}(\hat{y}_{k})+\frac{\eps_{1}}{3}\B_{\scrY}.
\]
Therefore, using Lemma \ref{lem:boundGradSmoothed}, we arrive at
\begin{align*}
\dist(0,\partial_{G}^{\eps_{2}}h(\hat{y}_{k})+\partial r_{1}(\hat{y}_{k}))^{2}&\leq \frac{6}{\alpha^{2}_{1}}\norm{y_{k}-\hat{y}_{k}}^{2}+3\norm{\nabla h_{\eta}(\hat{y}_{k})-\nabla h_{\eta}(y_{k})}^{2}_{\ast}+3\norm{\Delta W_{k+1}}^{2}_{\ast}+\frac{2\eps^{2}_{1}}{3}\\
&\leq \left(\frac{6}{\alpha^{2}_{1}}+3\frac{n\lip_{0}(h)^{2}}{\eta^{2}}\right)\norm{y_{k}-\hat{y}_{k}}^{2}+3\norm{\Delta W_{k+1}}^{2}_{\ast}+\frac{2\eps^{2}_{1}}{3}
\end{align*}
Next, we relate the auxiliary process $(\hat{y}_{k})_{k}$ to the stochastic sequence $(y_{k})_k$ generated by Algorithm \ref{alg:proxderivativefree}. To that end, observe that
\[
\norm{y_{k}-\hat{y}_{k}}\leq \norm{y_{k}-T_{\eta,\alpha_{1}}(y_{k})}+\norm{T_{\eta,\alpha_{1}}(y_{k})-\hat{y}_{k}}=\alpha_{1}\norm{\scrG_{\eta,\alpha_{1}}(y_{k})}+\alpha_{1}\norm{\Delta W_{k+1}}_{\ast}.  
\]
Combining this estimate and Lemma \ref{lem:boundGradSmoothed} with the penultimate display, we arrive at 
\begin{align*}
\dist(0,\partial_{G}^{\eps_{2}}h(\hat{y}_{k})+\partial r_{1}(\hat{y}_{k}))^{2}&\leq \left(12+\frac{6n\lip_{0}(h)^{2}}{\eta^{2}}\alpha_{1}^{2}\right)\norm{\scrG_{\eta,\alpha_{1}}(y_{k})}^{2}+\left(15+\frac{6n\lip_{0}(h)^{2}}{\eta^{2}}\alpha_{1}^{2}\right)\norm{\Delta W_{k+1}}^{2}_{\ast}+\frac{2\eps^{2}_{1}}{3}\\
&\leq \left(12+24\beta^{2}\right)\norm{\scrG_{\eta,\alpha_{1}}(y_{k})}^{2}+\left(15+24\beta^{2}\right)\norm{\Delta W_{k+1}}^{2}_{\ast}+\frac{2\eps^{2}_{1}}{3}\\
&\leq 18\norm{\scrG_{\eta,\alpha_{1}}(y_{k})}^{2}+21\norm{\Delta W_{k+1}}^{2}_{\ast}+\frac{2\eps^{2}_{1}}{3}
\end{align*}
Adopting a non-increasing step size regime in Algorithm \ref{alg:proxderivativefree}, we can leverage the monotonicity result of the prox-gradient mapping with respect to the step size, described in Appendix \ref{app:2}, so that for all $k\in\{0,1,\ldots,N\}$ 
\begin{equation}\label{eq:dist1}
\dist(0,\partial_{G}^{\eps_{2}}h(\hat{y}_{k})+\partial r_{1}(\hat{y}_{k}))^{2}\leq 18\norm{\scrG_{\eta,\alpha_{k}}(y_{k})}^{2}+21\norm{\Delta W_{k+1}}^{2}_{\ast}+\frac{2\eps^{2}_{1}}{3}.
\end{equation}
From these preparatory calculations, we can state the next relation between the complexity analysis in terms of the prox-gradient mapping (Corollary \ref{corr:Main}), and our definition of an $(\eps,\delta)$-stationary point (Definition \ref{def:GSstationarypoint}). 

\begin{theorem}
\label{th:ComplexityGoldstein}
Given $(\eps_{1},\eps_{2})$ positive parameters, let $(y_{k})_{k=0}^{N}$ be the stochastic process generated by Algorithm \ref{alg:proxderivativefree} with gradient estimator \eqref{eq:V}. Let Assumption \ref{ass:dombounded} together with all assumptions formulated in Corollary \ref{corr:Main2} hold true. Pick $\eta\in(0,\bar{\eta}]$ so that the gradient estimate of Proposition \ref{prop:epssmoothedgrad} for the given pair $(\eps_{1},\eps_{2})$. Let $(\hat{y}_{k})_{k=0}^{N}$ be the auxiliary process constructed recursively with 
\begin{equation}
\hat{y}_{0}=y_{0}\text{ and } \hat{y}_{k}=P_{\alpha_{1}}(y_{k},\hat{V}_{k+1})\quad\forall k=1,\ldots,N.
\end{equation}
If $\kappa:\Omega\to\{1,\ldots,N\}$ is the random variable with law defined in Theorem \ref{th:nonconvex1}, then for $N\geq 2$ chosen sufficiently large so that 
\begin{equation}\label{eq:NGold}
\frac{36\lip_{1}(h_{\eta})[\Psi_{\eta}(y_{1})-\Psi^{\Opt}_{\eta}]}{\beta\sqrt{N}}+\frac{\frac{228\cs^{2}}{\ca}(1+\log(N))}{\sqrt{N}}\leq\frac{\eps^{2}_{1}}{3}.
\end{equation}
Then, 
\[
\Ex\left[\dist(0,\partial_{G}^{\eps_{2}}h(\hat{y}_{\kappa})+\partial r_{1}(\hat{y}_{\kappa}))^{2}\right]\leq \eps^{2}_{1},
\]
i.e. the algorithm delivers an $(\eps_{1},\eps_{2})$ stationary point in the sense of Definition \ref{def:GSstationarypoint}.
\end{theorem}
\begin{proof}
Continuing from \eqref{eq:dist1} and using \eqref{eq:GNonconvex}, we readily deduce 
\begin{align*}
\Ex\left[\dist(0,\partial^{\eps_{2}}_{G}h(\hat{y}_{\kappa})+\partial r_{1}(\hat{y}_{\kappa}))^{2}\right]&\leq \frac{2\eps_{1}^{2}}{3}+21\sum_{k=1}^{N}\frac{\frac{\cs^{2}}{m_{k+1}}(2\alpha_{k}-\alpha^{2}_{k}\lip_{1}(h_{\eta}))}{\sum_{t=1}^{N}(2\alpha_{t}-\alpha^{2}_{t}\lip_{1}(h_{\eta}))}\\
&+18\left[\frac{4(\Psi_{\eta}(y_{1})-\Psi^{\Opt}_{\eta})}{\sum_{t=1}^{n}(2\alpha_{t}-\alpha^{2}_{t}\lip_{1}(h_{\eta}))}+\frac{\sum_{k=1}^{N}2\frac{\cs^{2}}{m_{k+1}}(4\alpha_{k}-\alpha^{2}_{k}\lip_{1}(h_{\eta}))}{\sum_{t=1}^{N}(2\alpha_{t}-\alpha^{2}_{t}\lip_{1}(h_{\eta}))}\right]\\
&\leq\frac{2\eps_{1}^{2}}{3}+\frac{72(\Psi_{\eta}(y_{1})-\Psi^{\Opt}_{\eta})}{\sum_{t=1}^{n}(2\alpha_{t}-\alpha^{2}_{t}\lip_{1}(h_{\eta}))}+\sum_{k=1}^{N}\frac{57\frac{\cs^{2}}{m_{k+1}}(4\alpha_{k}-\alpha^{2}_{k}\lip_{1}(h_{\eta}))}{\sum_{t=1}^{N}(2\alpha_{t}-\alpha^{2}_{t}\lip_{1}(h_{\eta}))}. 
\end{align*}
Choose $\alpha_{k}=\frac{2\beta}{\lip_{1}(h_{\eta})\sqrt{k}}$, so that $2\alpha_{k}-\alpha^{2}_{k}\lip_{1}(h_{\eta})\geq\alpha_{k}$ for all $k\geq 1$. Additionally, choosing $m_{k}=\ca\sqrt{k}$, and following the computations performed in Corollary \ref{corr:Main}, we continue the previous display as 
\begin{align*}
\Ex\left[\dist(0,\partial^{\eps_{2}}_{G}h(\hat{y}_{\kappa})+\partial r_{1}(\hat{y}_{\kappa}))^{2}\right]&\leq\frac{36\lip_{1}(h_{\eta})[\Psi_{\eta}(y_{1})-\Psi^{\Opt}_{\eta}]}{\beta\sqrt{N}}+\frac{2\eps^{2}_{1}}{3}+\frac{228\frac{\cs^{2}}{\ca}(1+\log(N))}{\sqrt{N}}.
\end{align*}
Choosing $N$ so large that \eqref{eq:NGold} holds, the thesis follows. 
\end{proof}

\begin{remark}
    Since $\lip_{1}(h_{\eta})=O(1/\eta)$ and Proposition~\ref{prop:epssmoothedgrad} shows that $\eta=O(\eps_{2})$, the implied iteration complexity by Theorem \ref{th:ComplexityGoldstein} is on the order of maginitude of $N^{-1/2}\sim \frac{\eps_{1}^{2}\eta}{3}$, so that $N\sim \frac{9}{\eps_{1}^{4}\eps_{2}^{2}}$. Choosing $\eps\equiv \eps_{1}=\eps_{2}$ therefore yields the leading order of $\eps^{-6}$ for the iteration complexity. 
\end{remark}

 \section{Numerical experiments}
 \label{sec:numerics}
%
In our numerical experiments, we consider the bilevel learning approach to inverse problems and specify two case studies: (i) regularization parameter selection; and (ii) optimal experimental design. Before reporting the numerical results in Section~\ref{sec:numericalresults} we specify how the methods are applied and validated in Section~\ref{sec:implementation}.

\subsection{Bilevel learning in inverse problems}\label{sec:bilevel_learning}
Our numerical experiments are designed for finite-dimensional linear inverse problem of reconstructing an unknown ground truth parameter $\bx^\dagger\in\scrX$ given noisy samples of the data $d\in\scrD$ defined as 
\[ 
d=\bK\bx^\dagger + Z. 
\]
where $\bK\in\R^{n_d\times n_x}$ is linear mapping from the parameter space $\scrX\equiv \mathbb R^{n_x}$ to the observation space $\scrD\equiv\R^{n_d}$. The $\R^{n_d}$-valued random variable $Z$ denotes observational noise. 
Due to ill-posedness of the reconstruction task we formulate our lower level problem in \eqref{eq:SBL} as regularized data misfit functional 
\begin{equation}\label{eq:linIP_lowerlevel}
\min_{\bx\in \R^{n_x}}\ g(\bx,\by,d),\quad g(\bx,\by,d)\eqdef \frac{1}2\norm{\bK x-d}^2 + \frac{\lambda}{2}\norm{\bL\bx}^2 + \tau {\mathrm{TV}}_\nu(\bx)\,,
\end{equation}
where $\bL\in\R^{n_x\times n_x}$ is a symmetric positive definite regularization matrix and $\lambda>0$ is the Tikhonov regularization parameter. In addition to the Tikhonov regularization, we consider a smoothed Total Variation regularization 
\[ 
{\mathrm{TV}}_\nu(\bx)\eqdef \sum_{i}\sqrt{|x_{i+1}-x_{i}|^2 + \nu^2}\,,
\]
when $\bx\in\R^{n_x}$ represents an one-dimensional signal, and
\[ 
{\mathrm{TV}}_\nu(\bx)\eqdef  \sum_{i,j}\sqrt{|x_{i+1,j}-x_{i,j}|^2+|x_{i,j+1}-x_{i,j}|^2 + \nu^2}\,,
\]
when $\bx\in\R^{{n_x}}$ represents a two-dimensional discretized image of size $\sqrt{n_x}\times \sqrt{n_x}$ px. Here, $\nu>0$ denotes the smoothing parameter and $\tau>0$ is the total variation regularization parameter. It is noteworthy to mention that bilevel learning problems with this particular lower level problem have also been studied in \cite{Ehrhardt:2021ab} and our numerical results in Section~\ref{ssec:denoising} will follow the setup described in that paper. In particular, as reported in \cite{Ehrhardt:2021ab}, the lower level problem \eqref{eq:linIP_lowerlevel} is $\mu$-strongly convex and $\beta$-smooth with
\[\mu \ge \lambda \cdot e_{\min}(\bL^2) \quad \text{and}\quad \beta\le \|\bK^\top \bK\| + \frac{\tau \partial}{\nu} + \lambda \|\bL^2\|\,, \]
where $e_{\min}(L^2)>0$ denotes the smallest eigenvalue of $\bL$ and $\partial>0$ is a constant arising due to the spatial discretization of the Total Variation. Hence, in order to solve the lower level problem we can implement a gradient descent scheme achieving a full control over the inexactness in the lower level solution later defined in Definition~\ref{def:inexact_lower}. More precisely, when implementing gradient descent with step size $\frac{1}{\beta}$ we achieve accuracy $\eps$ using the stopping criterion $\frac{\|\nabla_{1} g(\bx,\by,\xi_{2})\|^2}{\mu^2}\le \eps$. We employ this criterion in the implementation for solving the lower level problem using gradient descent.

In general terms, the bilevel learning approach for inverse problems can be formulated as statistical learning problem. In order to do so, we view the unknown parameter and the data as jointly distributed random variables $(X,D):\Omega \to \R^{n_x}\times \R^{n_d}$ defined as
\begin{equation*}
D(\omega) = \bK X(\omega) + Z(\omega),\quad \omega\in\Omega\,,
\end{equation*}
on the joint probability space $(\Omega,\scrF,\Pr)$. The random variables $(X,D)$ take the role of $(\xi_{1},\xi_{2})$ in our general stochastic bilevel optimization problem. The upper level problem in \eqref{eq:SBL} can be expressed as 
\begin{equation}\label{eq:linIP_upperlevel}
\min_{\by\in\scrY}\ F(x^\ast(\by,D),X)+r_1(\by),\quad F(x^\ast(\by,D),X)\eqdef \Ex_{\Pr}[\norm{\bx^\ast(\by,D)-X}^2]\, ,
\end{equation}
where for each sample $D(\omega)=d\in\mathbb R^{n_d}$ the vector $\bx^\ast(\by,d)\in\mathbb R^{n_x}$ denotes the solution of the lower level problem \eqref{eq:linIP_lowerlevel} for given parameters $\by$ (to be specified later) and $r_1:\scrY \to \R_+$ is possible additional regularization of the hyperparameter $\by$. Here, we view the training data consisting of independent ground truth realization $X$ and noisy observation $D$. The goal of the bilevel optimization problem \eqref{eq:linIP_upperlevel} is to choose parameters $\by$ such that the reconstruction via the lower level problem \eqref{eq:linIP_lowerlevel} is optimal over the considered (training) data distribution of $(X,D)$ in the mean-square sense. 

\subsubsection{Selecting the regularisation parameter} 
The reconstruction of the unknown parameter by solving \eqref{eq:linIP_lowerlevel} crucially depends on the choice of the regularization parameters $\lambda,\ \nu$ and the smoothing parameter $\nu$. Consequently, we let $\by=(\lambda,\tau,\nu)\in\scrY\equiv\mathbb R_+^3$. 
To estimate the signal 
based on observed data, the empirical risk minimization analogue of the stochastic optimization problem \eqref{eq:SBL} is often considered in the literature. We are given a sample $\{x_{1},\ldots,x_{N}\}\subset\R^{n_x}$ and $\{d_{1},\ldots,d_{N}\}\subset\R^{n_d}$, which are i.i.d realizations of the random element $(X,D)$. This setting can be included in \eqref{eq:SBL} and \eqref{eq:linIP_lowerlevel}-\eqref{eq:linIP_upperlevel} respectively by considering the 
empirical measure
\[
\Pr(A)\eqdef \frac{1}{N}\sum_{i=1}^{N}\delta_{(x_i,d_i)}(A)\qquad\forall A\in\scrB(\R^{n_x}\times \R^{n_d}). 
\]
over the labled data set $\{(x_1,d_1),\dots,(x_N,d_N)\}$. In this case, the upper level objective function reduces to the finite sum problem  
\[
h(\by)=\frac{1}{N}\sum_{i=1}^{N}F(\bx^{\ast}(\by,d_{i}),x_{i}) = \frac{1}{N}\sum_{i=1}^N \|x^\ast(\by,d_i)-x_i\|^2.
\]
This corresponds to a supervised learning problem where the labeled data set 
$\{(x_{1},d_{1}),\ldots,(x_{N},d_{N})\}$ is used to train $\bx\in\scrX$ as a function of the hyperparameter $\by\in\scrY$ and the data point $(x_{i},d_{i})$, so that 
\[
\bx^{\ast}_{i}(\by)\eqdef \bx^{\ast}(\by,d_i)\qquad \forall 
1\leq i\leq N .
\]
The optimization problem \eqref{eq:SBL} reduces then to the finite-sum composite problem 
\begin{equation}\label{eq:P}\tag{P}
\begin{split}
 & \min_{\by}\frac{1}{N}\sum_{i=1}^{N}\|\bx^{\ast}_{i}(\by)-x_{i}\|^2+r_{1}(\by) \\
 & \text{ s.t.: } \bx^{\ast}_{i}(\by)\in\argmin_{\bx\in\scrX}\{g(\bx,\by,d_{i})+r_{2}(\bx)\}\quad 1\leq i\leq N.
 \end{split}
 \end{equation}

\subsubsection{Optimal experimental design}\label{ex:OED}
In inverse problems, the forward model often depends on design parameters such as sensor placements or angle selection. In optimal experimental design (OED) one seeks to enhance the reconstruction by selecting the design parameters. 
Commonly, OED problems are formulated based on a pool of $n$ possible candidates for design. The goal is to find an optimal subset of $k$ observations from all candidates. While OED problems are often formulated in a Bayesian framework, aiming to maximize the expected information gain \cite{Lindley1956,Chaloner1995,Rainforth2024}, in our setting we formulate the OED problem as a stochastic bilevel optimization problem. A related approach has been proposed and studied in \cite{Ruthotto2018} with the motivation to incorporate state constraints.

Suppose that the variable $\by$ captures design parameters for the measurement process. The inner problem models the inverse problem for given design parameters. Following similar notation as the previous example, this could be included in \eqref{eq:SBL} and \eqref{eq:linIP_lowerlevel}-\eqref{eq:linIP_upperlevel} respectively
where the design parameters $\by\in\scrY$ may have a direct influence on the forward operator $\bK_{\by}$. In order to do so, we assume that the pool of design parameters is given by a fixed vector $\theta=(\theta_1,\dots,\theta_n)\in\R^n$. For a subset $J=\{j_1,\dots,j_k\}\subset \{1,\dots,n\}$ of size $|J|=k$ we define the projection $\theta_J:=\pi_J \theta := (\theta_{j_1},\dots,\theta_{j_k})\in\R^k$. The linear forward model is assumed to be depending on a selection of $k$ design parameters such that we may write $\bK_J = \bK(\pi_J\theta)\in\R^{n_d\times n_x}$ for any choice $J\subset \{1,\dots,n\}$. More precisely, we assume that the mapping $J\mapsto \bK_J$ is measurable and given a subset $J$, the inverse problem reads as
\[d = \bK_J \bx + Z\,, \]
where $Z$ may also depend on $J$. For example, one can view the measurements $d=\bK\bx$ as consisting of $n$ atomic measurements, corresponding to (blocks of) rows of $\bK$. In our OED formulation, we then aim to find a set of $k \leq n$ experiments that lead to the best possible reconstruction over a training data set. 

To model the OED, we introduce the set of possible policies $\scrP:=\{p=(p_1,\dots,p_n)\in\R^n: \sum_{i=1}^n p_i = 1,\ p_i\ge0\}$ which describe the probabilities of selecting a particular (block of) rows from $\bK$. The parameter for the upper level problem now consists of this policy, 
i.e., $\by \in \scrP$. In this case, we view the unknown parameter $X$ and the noisy observations $D$ as jointly distributed random variable $(X,D):\Omega\to\R^{n_x}\times \R^{n_d}$ defined by
\begin{equation}\label{eq:data_distr}
 D(\omega) = \bK_{J(\omega)} X(\omega) + Z(\omega),\quad \omega\in\Omega\,,
\end{equation}
where the random variables $X$, $Z:\Omega\to\R^{n_d}$ and $J:\Omega\to \bigtimes\limits_{i=1}^k \{1,\dots,n\} $ are assumed to be independent. The solution map $\bx^*(\by,d)$ now denotes the solution to \eqref{eq:linIP_lowerlevel} with $\bK \equiv \bK_J$ and $d$ drawn according the data-generating process outlined above. Note that the forward operator $\bK_J$ consists of $k$ (blocks of) rows, drawn at random according to the policy $p$.

\subsection{Implementation and validation}\label{sec:implementation}
In both examples, we have implemented Algorithm~\ref{alg:proxderivativefree} with inexact lower level solution. Following our theoretical findings 
we solve the lower level problem up to accuracy $\beta_k = \frac{\beta_0}{\sqrt{k}}$, $\beta_0>0$ and increase the batch size of the random gradient estimator \eqref{eq:Vinexact} by $m_k = \sqrt{k}\cdot m_0$, $m_0\in\mathbb N$. We adopt the step-size policy $\alpha_k = \frac{\alpha_0}{\sqrt{k}}$, $\alpha_0>0$. The smoothing parameter is fixed at level $\eta=0.01$ for the first example and decreased by the schedule $\eta_k = 1/\sqrt{k}$ in the second example.  In order to numerically illustrate the convergence of the generated trajectory, we plot the summation over the random operator $\tilde{\scrG}_{\eta,k}^{\beta_k}$ scaled by the step-size policy, i.e.~we demonstrate that
\[ \Delta_k := \sum_{s=1}^k \alpha_s \tilde{\scrG}_{\eta,s}^{\beta_s}\]
remains bounded. To verify the generalization properties of the method, we generate a validation data set independent of the data set applied in the application of Algorithm~\ref{alg:proxderivativefree}, defined as i.i.d.~sample $(x^{i,{\mathrm{val}}},d^{i,{\mathrm{val}}})_{i=1}^{m_{\mathrm{val}}}$, $m_{\mathrm{val}}\in\mathbb N$, of $(X,D)$ defined in \eqref{eq:data_distr}. As a result we plot the normalized empirical errors in the upper level 
\begin{equation}\label{eq:validation_error}
e_i(\by) := \frac{\|\bx^\ast(\by,d^{i,{\mathrm{val}}})- x^{i,{\mathrm{val}}}\|}{\|x^{i,{\mathrm{val}}}\|}\quad \forall i\in\{1,\dots,m_{\mathrm{val}}\}.
\end{equation}
The lower level solution $\bx^\ast(\by,d^{i,{\mathrm{val}}})$ is obtained via the gradient descent method with tolerance $\beta = 10^{-7}$. In this visualization, we compare the generalization error for different choices of $\by$ with the resulting learned parameters by Algorithm~\ref{alg:proxderivativefree}.


\subsection{Numerical results}\label{sec:numericalresults}

\subsubsection{One-dimensional signal denoising}\label{ssec:denoising}

In the first experiment, we consider a simple one-dimensional image denoising problem, inspired by \cite{Ehrhardt:2021ab}. The goal is to reconstruct a noisy one dimensional piece-wise constant signal. We represent the signal as random vector $X =(X({t_1}),\dots, X({t_{n_x}}))^\top\in \R^{n_x}$, with $n_x = 256$ sample points, corrupted by Gaussian white noise
\[D({t}) = X(t) + \sigma\,Z_t\,,\quad t\in\{t_1,\dots,t_{n_x}\} \]
where $(Z_t)_{t\in\{t_1,\dots,t_{n_x}\}}$ are independent and identically distributed with $Z_1\sim\Normal(0,1)$, $\sigma=\sqrt{0.001}$. In our experiments, we set for $t_i = \frac{i}{n_x}$, $i=1,\dots,n_x$, 
\[ X({t_i},\omega)= \mathbf{1}_{[C(\omega),R(\omega)]}(t_i)\,,\]
where $C,R$ are two independent uniformly distributed random variables with $C\sim\Uniform([\frac18,\frac14])$ and $R\sim\Uniform([\frac38,\frac78])$. 

In our implementation, we have introduced the parametrization $\lambda(y^{(1)}) = 10^{y^{(1)}}$, $\tau(y^{(2)}) = 10^{y^{(2)}}$ and $\nu(y^{(3)}) = 10^{y^{(3)}}$ with additional constrain $\by = (y^{(1)},y^{(2)},y^{(3)}) \in [-7,7]^3\subset\mathbb R^3$. The resulting proximal operator becomes a projection operator into $[-7,7]^3$.  Moreover, we apply a second order regularization matrix $\bL^2 = 0.01^2\Delta^{-1}$ for the Tikhonov regularization in the lower level problem \eqref{eq:linIP_lowerlevel}, where $\Delta$ denotes the (discretized) Laplace operator. In addition, we include a regularized upper level defined as
\begin{equation*}
\min_{\by\in[-7,7]^3}\ \Ex_{\Pr_{0}}[\|\bx^\ast(\by,D)-X\|^2] + 10^{-6} \left(\frac{\|\bK^\top \bK\| + \frac{\tau(y^{(2)}) \partial}{\nu(y^{(3)})} + \lambda(y^{(1)}) \|\bL^2\|}{\lambda(y^{(1)}) e_{\min}(\bL^2)}\right)^2\, ,
\end{equation*}
to avoid too large condition numbers of the lower level problem. Finally, we set $\alpha_0=1$, $\beta_0=0.01$ and $m_0=1$, and terminate Algorithm~\ref{alg:proxderivativefree} after $N=700$ iterations.

In Figure~\ref{fig:1} (a)-(c) we plot the resulting regularization parameters generated by Algorithm~\ref{alg:proxderivativefree}, where we observe that all three parameters converge. This result can also be observed from Figure~\ref{fig:1} (d), where we demonstrate that the summation over the random operators $\tilde{\scrG}_{\eta,k}^{\beta_k}$ remains bounded. The resulting reconstruction of the signal using the learned regularization parameters for solving the lower level problem \eqref{eq:linIP_lowerlevel} is plotted in Figure~\ref{fig:2} (b).  As comparison, in Figure~\ref{fig:2} (c)-(f), we plot the reconstructions of the signal using different regularization parameters chosen by hand. In all four cases, we have chosen a smoothing parameter $\nu = 10^{-3}$. The comparison of the reconstruction already suggest that our learned regularization parameter using Algorithm~\ref{alg:proxderivativefree} outperforms the fixed regularization parameters. This suggestion is further demonstrated in Figure~\ref{fig:3} where we compare the generalization error \eqref{eq:validation_error} over validation data set independent of the training data set.

\begin{figure}[htb!]
\includegraphics[width=1\textwidth]{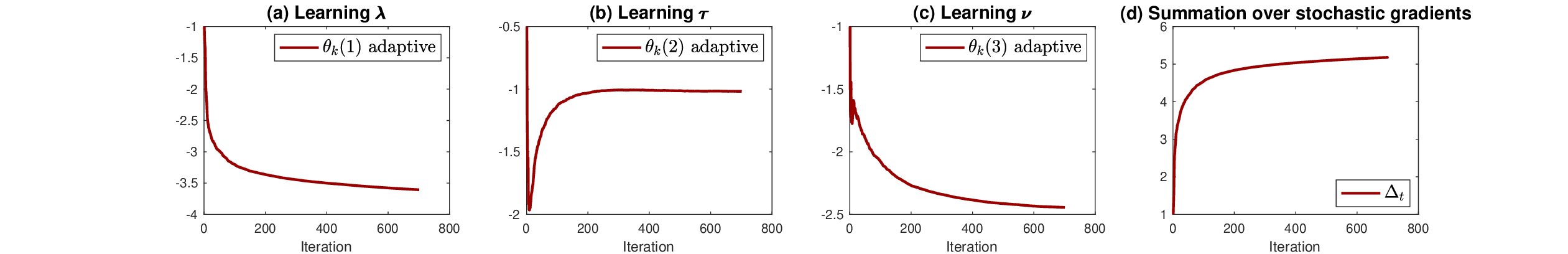}
\caption{(a)-(c) Learned regularization parameters $(y_k^{(1)},y_k^{(2)},y_k^{(3)})$ resulting from Algorithm~\ref{alg:proxderivativefree} and (d) summation over the random operators $\tilde{\scrG}_{\eta,t}^{\beta_t}$.} \label{fig:1}
\end{figure}

\begin{figure}[htb!]
\includegraphics[width=1\textwidth]{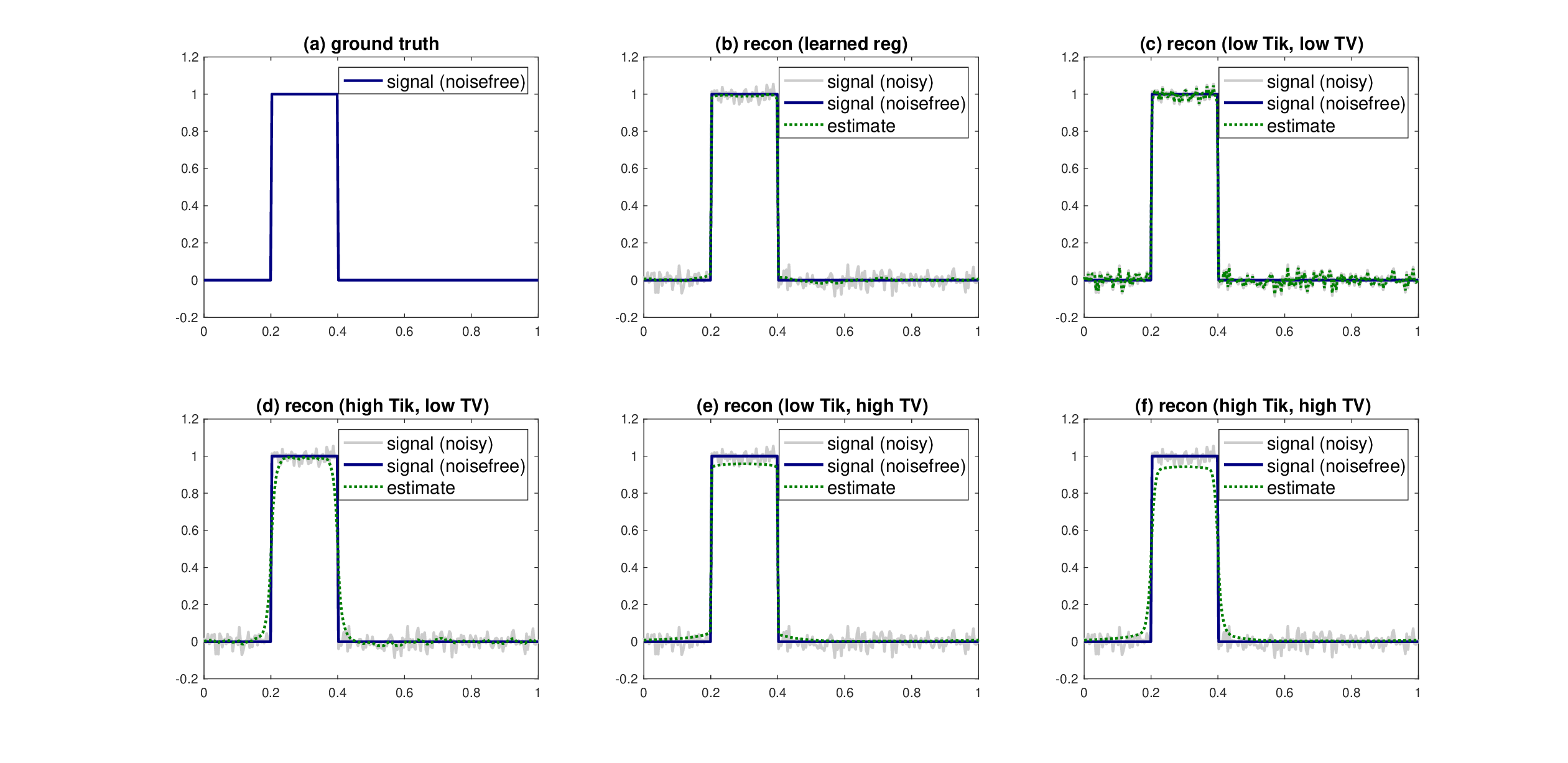}
\caption{(a) Ground truth signal and (b) reconstruction of the signal using  the learned regularization parameters $(\lambda,\tau,\nu)(y_N)$ after $N=300$ iterations. As comparison we show the reconstruction (c) using low Tikhonov regularization with $\lambda = 10^{-3}$ and low TV regularization with $\tau = 10^{-3}$, (d) using high Tikhonov regularization with $\lambda = 10^{-1}$ and low TV regularization with $\tau = 10^{-3}$, (e) using low Tikhonov regularization with $\lambda = 10^{-3}$ and high TV regularization with $\tau = 1$, and (f) using high Tikhonov regularization with $\lambda = 10^{-1}$ and high TV regularization with $\tau = 1$. In (c)-(f) we have fixed the smoothing parameter $\nu=10^{-3}$.} \label{fig:2}
\end{figure}

\begin{figure}[htb!]
\centering \includegraphics[width=0.4\textwidth]{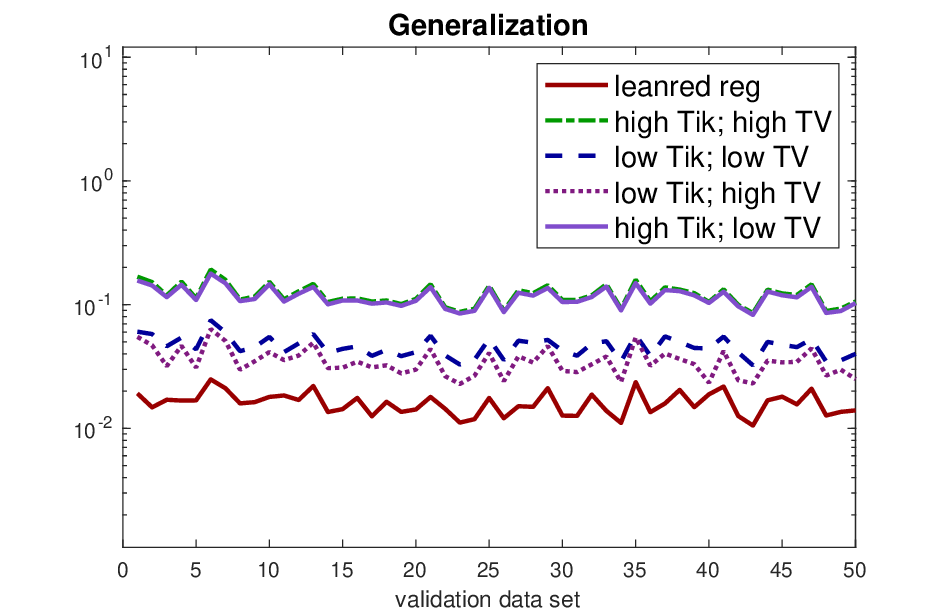}
\caption{Pointwise generalization error in the upper level $e_i(\lambda,\tau,\nu)$ over the validation data set $(x_i^{\mathrm{val}},d_i^{\mathrm{val}})_{i=1}^{m_{\mathrm{val}}}$, $m_{\mathrm{val}}=50$.  We plot the errors for the different choices of regularization parameters from Figure~\ref{fig:2}.}\label{fig:3}
\end{figure}

\subsubsection{Image reconstruction based on  the radon transform}\label{ex:radon}
In X-ray tomography, the forward operator $\bK_\theta$ is a discretization of the Radon transform \cite{Hansen2021}, where data are collected at various angles $\theta \in [0,\pi)$. The unknown $\bx$ represents a 2D image, and the measurements $d(\theta) = \bK_\theta \bx$ for one angle represents the line integrals of that image along straight lines at angle $\theta$. Collecting a large number of angles $\theta \in [0,\pi)$ leads to a well-posed inverse problem and generally yields a good reconstruction. For practical applications it is of interest to reduce the number of angles, dictating the use of additional regularization to fill in the missing information.

In the following experiment, we assume that we are allowed to pick $k=6$ angles out of a pool of $n = 64$ possible angles $\theta_i = \frac{(i-1)\pi}{n}$. The goal is to reconstruct images of size $64\times 64$ px given the noisy measurements $(d(\theta_{j_1}),\dots,d(\theta_{j_6}))\in\R^6$ constructed by $d(\theta_{j_i}) = \bK_{\theta_{j_i}}\bx + Z_{i}$, $i=1,\dots,6$ where $(Z_i)_{i=1,\dots,6}$ are independent and identical distributed according to $\Normal(0,0.01^2)$. Our set of images consist of randomly generated triangles of varying size, rotation in the space and varying gray levels ranging from $0.5$ to $1$. The angles and direction of the triangles are kept fixed. In Figure~\ref{fig:ex2_samples}, we show i.i.d.~realization of $16$ different images. 
\begin{figure}[htb!]
\centering \includegraphics[width=0.6\textwidth]{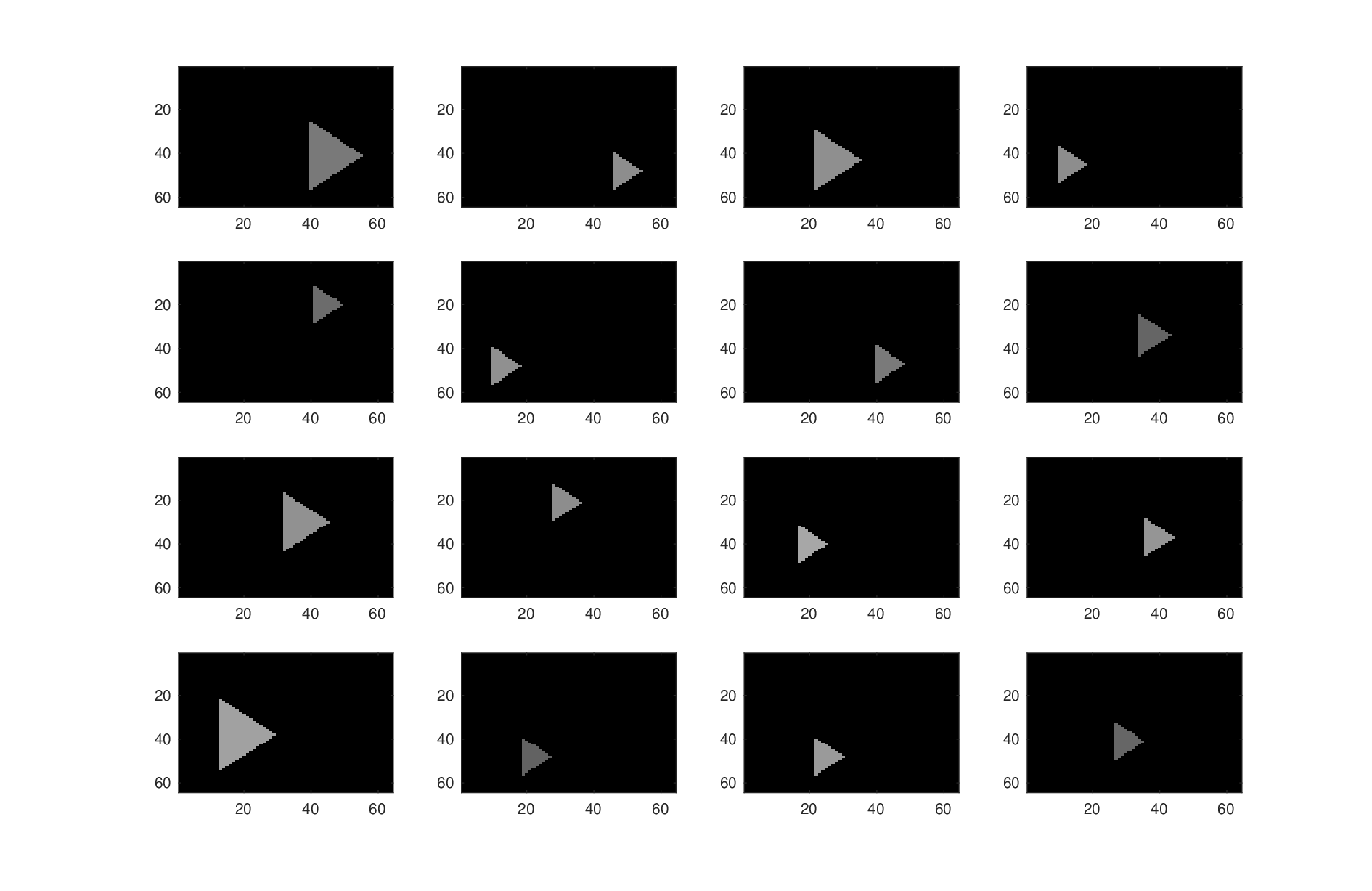}
\caption{Realizations of the random triangles in Example~\ref{ex:radon}}\label{fig:ex2_samples}
\end{figure}

In order to enhance the reconstruction accuracy we have implemented the OED problem of choosing the best possible policy over the set of all possible angles \cite{10572344}. The regularization parameters are again parametrized as $\lambda(y^{(1)})=10^{y^{(1)}}$, $\tau(y^{(2)}) = 10^{y^{(2)}}$ and $\nu(y^{(3)}) = 10^{y^{(3)}}$ with the additional constraint $(y^{(1)},y^{(2)},y^{(3)})\in[-7,7]^3$. Moreover, we set the regularization matrix $\bL =\Id$ for the Tikhonov regularization in the lower level problem \eqref{eq:linIP_lowerlevel}. In addition, we incorporate state constrain to the lower level solution forcing the solutions to remain non-negative which is implemented using a projected gradient method. 
For the parametrization of the policy we've used a soft-max parametrization such that the probability distributions are defined as 
\[ p_i = \frac{\exp(\theta_i)}{\sum_{j=1}^n\exp(\theta_j)},\quad i=1,\dots,n\,.\]
To compromise the notation, we define $(y^{(4)},\dots y^{(n+3)}):=(\theta_1,\dots,\theta_n)$ such that our upper level problem \eqref{eq:linIP_upperlevel} is a minimization problem over a space $\scrY$ of dimension $n+3 = 67$. Given a policy $p=(p_1,\dots,p_n)\in\scrP$, the data $D$ is generated by
\[D(\omega) = \bK_{J(\omega)} X(\omega) + Z(\omega)\,, \]
where $J:\Omega\to \bigtimes\limits_{i=1}^k \{1,\dots,n\}$ is a random variable generating $k$ samples of the policy $p$ without replacement, $\bK_J=(\bK_{\theta_{j_1}},\dots, \bK_{\theta_{j_6}})$ denotes the forward map given $J=(j_1,\dots,j_6)\in\{1,\dots,n\}^6$ realized angles and $Z=(Z_1,\dots,Z_6)$ denotes the i.i.d.~measurement noise with $Z_i\sim\Normal(0,0.01^2)$. As discussed above, we assume that all random variables $J$, $X^\dagger$ and $Z$ are independent.

In our numerical implementation, we have chosen the uniform policy $(p_1,\dots,p_n) = (1/n,\dots,1/n)$ as initial condition. The same policy is used as comparison in our validation over the validation data set. Algorithm~\ref{alg:proxderivativefree} with inexact lower level solution is applied with $\alpha_0 = 0.2$, $\beta_0=0.1$ and $m_0=1$, and terminated after $N=2000$ iterations. The generalization performance is illustrated in Figure~\ref{fig:ex2_generalization}, where we have applied various configurations of regularization parameters together with the uniform policy. Among fixed choices of regularization parameters, we have also implemented the bilevel learning approach for selecting the regularization parameters $(\lambda,\tau,\nu)$ with a fixed uniform policy. Overall, we observe a significant improvement by applying our learned policy. The resulting reconstructions for the different choices of regularization parameters and policies are shown in Figure~\ref{fig:ex2_recon_collection}. These reconstructions further demonstrate the significant improvement through the proposed OED approach based on the stochastic bilevel optimization problem.

\begin{figure}[htb!]
\centering \includegraphics[width=1\textwidth]{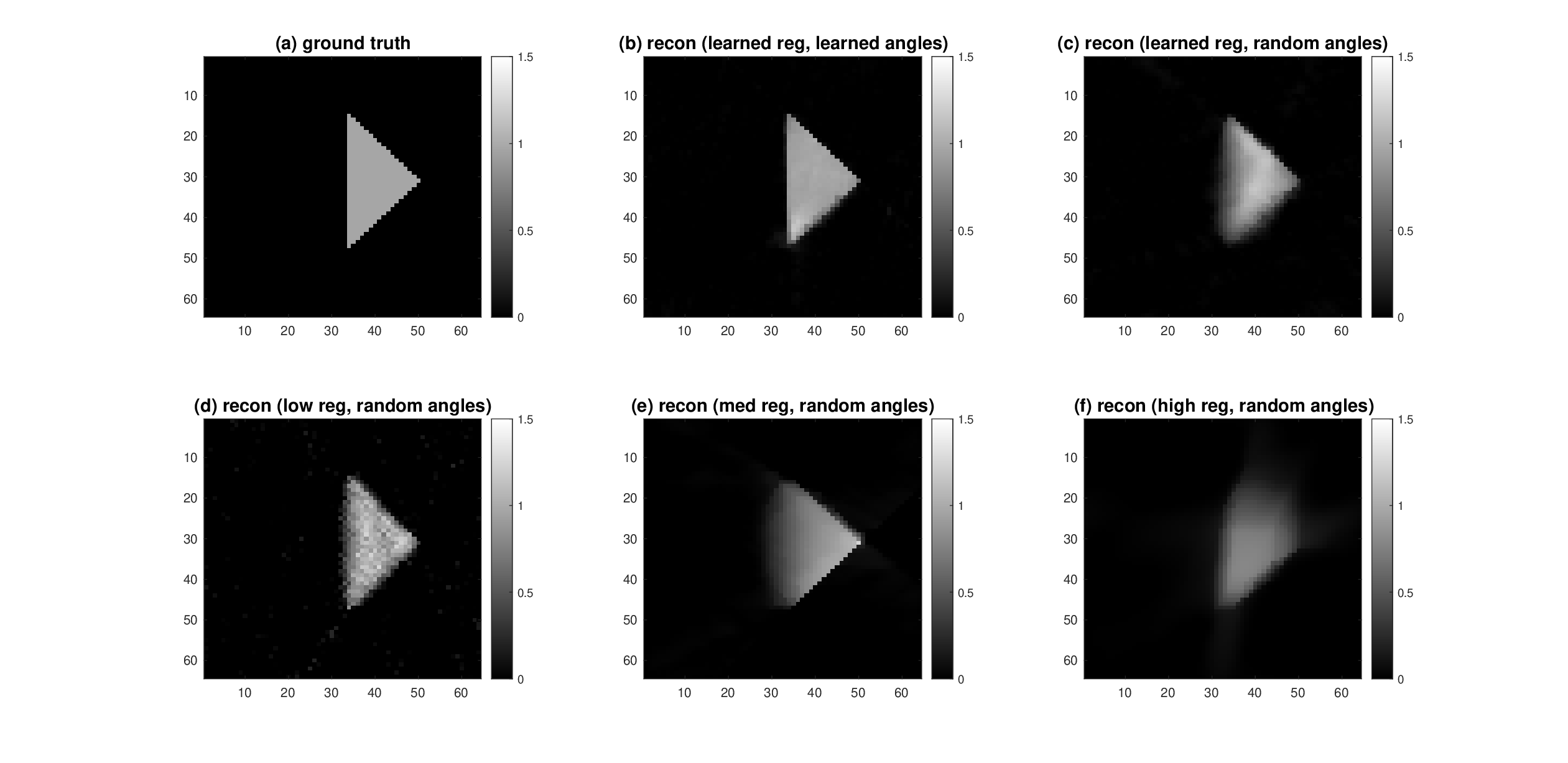}
\caption{(a) Ground truth image and (b) reconstruction of the image using the learned regularization parameters $(\lambda,\tau,\nu)(\by_N)$ and the learned policy $p(\by_N)$ after $N=2000$ iterations. As comparison we show the reconstruction of the image (c) using the learned regularization parameters $(\lambda,\tau,\nu)(\by_N)$ after $N=2000$ iterations and a fixed uniform policy, (d) using low regularization with $\lambda = 10^{-9}$, $\tau = 10^{-9}$, $\nu = 10^{-2}$ and uniform policy, (e) using medium regularization with $\lambda = 10^{-3}$, $\tau = 10^{-3}$, $\nu = 10^{-2}$ and uniform policy, (f) using high regularization with $\lambda = 10^{-2}$, $\tau = 10^{-2}$, $\nu = 10^{-2}$ and uniform policy. }\label{fig:ex2_recon_collection}
\end{figure}

\begin{figure}[htb!]
\centering \includegraphics[width=0.4\textwidth]{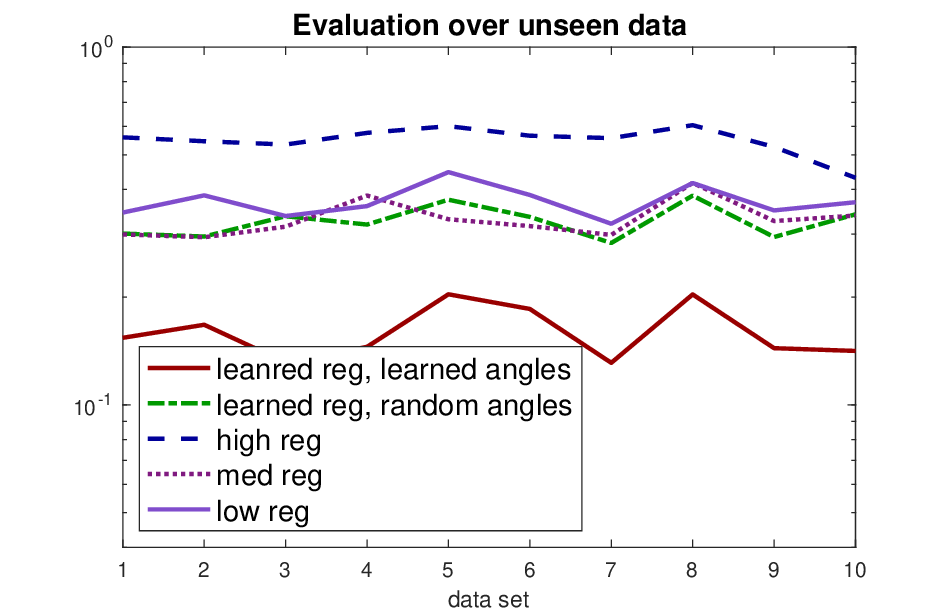}
\caption{Pointwise generalization error in the upper level $e_i(p,\lambda,\tau,\nu)$ over the validation data set $(x_i^{\mathrm{val}},d_i^{\mathrm{val}})_{i=1}^{m_{\mathrm{val}}}$, $m_{\mathrm{val}}=10$.  We plot the errors for the different choices of regularization parameters from Figure~\ref{fig:ex2_recon_collection}.}\label{fig:ex2_generalization}
\end{figure}

%
%
%
%

 \section{Conclusion}
 \label{sec:conclusion}
%
In this paper we've studied a zeroth-order gradient method for a particular class of stochastic bilevel programs which arise naturally in data-driven learning of inverse problems. Our complexity estimates adapt to smoothing and inexact solutions of the lower level problem. Our theoretical and numerical results display the favourable properties of our scheme. In future work, we plan to continue this line of research along the following directions: 
\begin{itemize}
\item Higher-order numerical methods: The merit function employed in this paper is a stationary point. In non-convex optimization, an important question is whether our method is able to avoid saddle-points. For this, we plan to develop stochastic Newton methods, employing derivative-free gradient estimation strategies, as done in this paper.
\item Weakening the Lipschitz continuity assumptions of the hyperobjective. Interesting recent results in this direction are reported in \cite{Lei:2024aa}. 
\item Construction of the random estimator: In this paper we adopt a Monte-Carlo approach to estimate the directional deriviative using iid Gaussian directions. It would be interesting to include more structure in this sampling approach. Quasi- or Multi-level Monte Carlo approaches would be interesting new stochastic simulation approaches to reduce the computational costs \cite{Giles_2015}. 
\end{itemize} 

\paragraph*{Acknowledgments}
 The first author thanks the FMJH Program Gaspard Monge for optimization and operations research and their interactions with data science for financial support.
 \begin{appendix}
 \section{Properties of the Gaussian smoothing}
 \label{app:Smoothing}
%

%
Let $\scrE$ be a finite-dimensional real vector space, and define $M_{p}\eqdef \Ex[\norm{U}^{p}]$.
 \begin{lemma}[\cite{Nesterov:2017wm}, Lemma 1]
 \label{lem:Gaussbound}
 We have $M_{0}=1, M_{2}=n$ and for $p\in[0,2]$, 
 \begin{equation}\label{eq:Gaussboundlow}
 M_{p}\leq n^{p/2}. 
 \end{equation}
 If $p\geq 2$, then 
 \begin{equation}\label{eq:Gaussboundhigh}
 n^{p/2}\leq M_{p}\leq (p+n)^{p/2} 
 \end{equation}
 \end{lemma}
 
\begin{proof}[Proof of Lemma \ref{lem:Liph}]
For all $\by_{1},\by_{2}\in\scrY$, we have
\begin{align*}
\abs{h_{\eta}(\by_{1})-h_{\eta}(\by_{2})}&=\abs{\Ex_{\Pr_{1}}[h(\by_{1}+\eta U)]-\Ex_{\Pr_{1}}[h(\by_{2}+\eta U)]}\\
&=\abs{\Ex_{\Pr_{1}}[h(\by_{1}+\eta U)-h(\by_{2}+\eta U)]}\\
&\leq \Ex_{\Pr_{1}}\left[\abs{h(\by_{1}+\eta U)-h(\by_{2}+\eta U)}\right]\\
&\leq \lip_{0}(h)\norm{\by_{1}-\by_{2}}.
\end{align*}
\end{proof}

\begin{proof}[Proof of Lemma \ref{lem:boundfunctions}]
For any $\by\in\scrY$ we have
\begin{align*}
\abs{h_{\eta}(\by)-h(\by)}\leq \Ex_{\Pr_{1}}\left[\abs{h(\by+\eta U)-h(\by)}\right]\leq \eta\lip_{0}(h)\Ex_{\Pr_{1}}\left[\norm{U}\right]=\eta\lip_{0}(h)\sqrt{n}. 
\end{align*}
\end{proof}

\begin{proof}[Proof of Lemma \ref{lem:boundGradSmoothed}]
Using the formula \eqref{eq:smoothing}, for any $\by\in\scrY$, we can directly differentiate under the integral to obtain 
\begin{align*}
\nabla h_{\eta}(\by)&=\frac{1}{\Upsilon\eta^{n}}\int_{\scrY}h(\bz)\exp\left(-\frac{1}{2\eta^{2}}\norm{\bz-\by}^{2}\right)\frac{B(\bz-\by)}{\eta^{2}}\dif\bz\\
&=\frac{1}{\Upsilon}\int_{\scrY}\frac{1}{\eta}h(\by+\eta\bu)\exp\left(-\frac{1}{2}\norm{\bu}^{2}\right)B\bu\dif\bu\\
&=\Ex_{\Pr_{1}}\left[\frac{h(\by+\eta U)-h(\by)}{\eta}BU\right]\\
&=\Ex_{\Pr_{1}}\left[\frac{h(\by+\eta U)}{\eta}BU\right].
\end{align*}
Now let $\by_{1},\by_{2}\in\scrY$ so that 
\begin{align*}
\norm{\nabla h_{\eta}(\by_{1})-\nabla h_{\eta}(\by_{2})}_{\ast}&\leq \Ex_{\Pr_{1}}\left[\abs{\frac{h(\by_{1}+\eta U)-h(\by_{2}+\eta U)}{\eta}}\norm{BU}_{\ast}\right]\\
&\leq \lip_{0}(h)\frac{\norm{\by_{1}-\by_{2}}}{\eta}\Ex_{\Pr_{1}}\left[\norm{U}\right]\\
&\leq\lip_{0}(h)\frac{\norm{\by_{1}-\by_{2}}}{\eta}\sqrt{n}
\end{align*}
where the last inequality uses \cite[][Lemma 1]{Nesterov:2017wm}. To obtain the bound on the gradient norm, we continue from the first relation, showing that 
\begin{align*}
\norm{\nabla h_{\eta}(\by)}^{2}_{\ast}&\leq \Ex_{\Pr_{1}}\left[\abs{\frac{h(\by+\eta U)-h(\by)}{\eta}}^{2}\norm{BU}^{2}_{\ast}\right]\\
&\leq \lip_{0}(h)^{2}\Ex_{\Pr_{1}}\left[\norm{U}^{2}\cdot\norm{BU}^{2}_{\ast}\right]\\
&=\lip_{0}(h)^{2}\Ex_{\Pr_{1}}\left[\norm{U}^{4}\right]\leq \lip_{0}(h)^{2}(4+n)^{2}.
\end{align*}
The last equality uses again \cite[][Lemma 1]{Nesterov:2017wm}.
\end{proof}
 \section{Technical Proofs}
 \label{app:1}
%

\subsection{Proof of Lemma \ref{lem:errorbound}}
Given $\by\in\scrY$, we use the law of iterated expectations to compute
\begin{align*}
\Ex_{\Pr}\left[\hat{V}_{\eta,m}(\by)\right]&=\Ex_{\Pr}\left[\frac{1}{m}\sum_{i=1}^{m}\hat{\nabla}_{(U^{i},\eta)}H(\by,\xi^{i})\right]\\
&=\Ex_{\Pr}\left[\frac{1}{m}\sum_{i=1}^{m}\Ex_{\Pr}\left(\hat{\nabla}_{(U^{i},\eta)}H(\by,\xi^{i})\vert \sigma(U^{i})\right)\right]\\
&=\frac{1}{m}\sum_{i=1}^{m}\Ex_{\Pr}\left[\frac{h(\by+\eta U^{i})-h(\by)}{\eta}BU^{i}\right]\\
&=\nabla h_{\eta}(\by)
\end{align*}
where the last equality uses eq. \eqref{eq:GradImplicit}. For the second bound, observe that 
\begin{align*}
\Ex_{\Pr}\left[\norm{\hat{V}_{\eta,m}(\by)}^{2}_{\ast}\right]&=\Ex_{\Pr}\left[\norm{\frac{1}{m}\sum_{i=1}^{m}\left(\hat{\nabla}_{(U^{i},\eta)}H(\by,\xi^{i})-\nabla h_{\eta}(\by)\right)+\nabla h_{\eta}(\by)}^{2}_{\ast}\right]\\
&=\frac{1}{m^{2}}\Ex_{\Pr}\left[\norm{\sum_{i=1}^{m}\left(\hat{\nabla}_{(U^{i},\eta)}H(\by,\xi^{i})-\nabla h_{\eta}(\by)\right)}^{2}_{\ast}\right]+\norm{\nabla h_{\eta}(\by)}^{2}_{\ast}.
\end{align*}
Define the centered random variable $X_{i}\eqdef \hat{\nabla}_{(U^{i},\eta)}H(\by,\xi^{i})-\nabla h_{\eta}(\by)$ for $1\leq i\leq m$, to obtain an i.i.d collection of zero-mean random variables in $\scrY^{\ast}$. Therefore, we can continue from the last line of the previous display by noting that 
\begin{align*}
\Ex_{\Pr}\left[\norm{\sum_{i=1}^{m}X_{i}}^{2}_{\ast}\right]&=\Ex_{\Pr}\left[\inner{B^{-1}\sum_{i=1}^{m}X_{i},\sum_{i=1}^{m}X_{i}}\right]\\
&=\sum_{i,j}\Ex_{\Pr}\left[\inner{B^{-1}X_{i},X_{j}}\right]=\sum_{i=1}^{m}\Ex_{\Pr}\left[\inner{B^{-1}X_{i},X_{i}}\right]\\
&=\sum_{i=1}^{m}\Ex_{\Pr}[\norm{X_{i}}^{2}_{\ast}].
\end{align*}
Since $\Ex_{\Pr}[\norm{X_{i}}^{2}_{\ast}]=\Ex_{\Pr}\left[\norm{\hat{\nabla}_{(U,\eta)}H(\by,\xi)}^{2}_{\ast}\right]-\norm{\nabla h_{\eta}(\by)}^{2}_{\ast}$, it follows 
\[
\Ex_{\Pr}\left[\norm{\hat{V}_{\eta,m}(\by)}^{2}_{\ast}\right]\leq\frac{1}{m} \Ex_{\Pr}\left[\norm{\hat{\nabla}_{(U,\eta)}H(\by,\xi)}^{2}_{\ast}\right]+(1-\frac1m)\|\nabla h_\eta(\by)\|_\ast^2.
\]
\begin{align*}
\Ex_{\Pr}\left[\norm{\hat{\nabla}_{(U,\eta)}H(\by,\xi)}^{2}_{\ast}\right]&=\Ex_{\Pr}\left[\norm{\frac{H(\by+\eta U,\xi)-H(\by,\xi)}{\eta}BU}^{2}_{\ast}\right]\\
&=\Ex_{\Pr}\left[\abs{\frac{H(\by+\eta U,\xi)-H(\by,\xi)}{\eta}}^{2}\cdot\norm{BU}^{2}_{\ast}\right]\\
&\leq \Ex_{\Pr}\left[\lip_{0}(H(\cdot,\xi))^{2}\norm{U}^{4}\right]\stackrel{\textnormal{Lemma \ref{lem:Gaussbound}}}{\leq} \abs{\lip_{0}(H(\cdot,\xi))}_{2}^{2}(4+n)^{2} .
\end{align*} 

\subsection{Proof of Lemma \ref{lem:Psiprogress}}
\label{proof:Lemma5.1}
The optimality condition for the iterate $y_{k+1}$ gives 
$$
B\left(\frac{y_{k}-y_{k+1}}{\alpha_{k}}\right)\in \hat{V}_{k+1}+\partial r_{1}(y_{k+1}).
$$
This means that there exists $\rho_{k+1}\in\partial r_{1}(y_{k+1})$ satisfying 
$$
\hat{V}_{k+1}+\rho_{k+1}=B\left(\frac{y_{k}-y_{k+1}}{\alpha_{k}}\right).
$$
Since $r_{1}(\cdot)$ is convex, the convex subgradient inequality gives for all $u\in\scrY$,
\begin{align}
r_{1}(u)&\geq r_{1}(y_{k+1})-\inner{\hat{V}_{k+1}-B\left(\frac{y_{k}-y_{k+1}}{\alpha_{k}}\right),u-y_{k+1}} \notag \\
&=r_{1}(y_{k+1})-\inner{\hat{V}_{k+1},u-y_{k+1}}+\frac{1}{\alpha_{k}}\inner{B(y_{k+1}-y_{k}),y_{k+1}-u} \label{eq:ineq_r}
\end{align}
Set $u=y_{k}$ to obtain 
\begin{align*}
r_{1}(y_{k})&\geq r_{1}(y_{k+1})-\inner{\hat{V}_{k+1},y_{k}-y_{k+1}}+\frac{1}{\alpha_{k}}\norm{y_{k+1}-y_{k}}^{2}\\
&=r_{1}(y_{k+1})-\alpha_{k}\inner{\hat{V}_{k+1},\tilde{\scrG}_{\eta,\alpha_{k}}(y_{k})}+\alpha_{k}\norm{\tilde{\scrG}_{\eta,\alpha_{k}}(y_{k})}^{2}.
\end{align*}

The descent property \eqref{eq:descent} for $h_{\eta}\in\bC^{1,1}(\scrY)$ gives 
\begin{align*}
h_{\eta}(y_{k+1})&\leq h_{\eta}(y_{k})+\inner{\nabla h_{\eta}(y_{k}),y_{k+1}-y_{k}}+\frac{\lip_{1}(h_{\eta})}{2}\norm{y_{k+1}-y_{k}}^{2}\\
&=h_{\eta}(y_{k})-\alpha_{k}\inner{\nabla h_{\eta}(y_{k}),\tilde{\scrG}_{\eta,\alpha_{k}}(y_{k})}+\frac{\alpha^{2}_{k}\lip_{1}(h_{\eta})}{2}\norm{\tilde{\scrG}_{\eta,\alpha_{k}}(y_{k})}^{2}\\
&=h_{\eta}(y_{k})-\alpha_{k}\inner{\hat{V}_{k+1},\tilde{\scrG}_{\eta,\alpha_{k}}(y_{k})}+\frac{\alpha^{2}_{k}\lip_{1}(h_{\eta})}{2}\norm{\tilde{\scrG}_{\eta,\alpha_{k}}(y_{k})}^{2}\\
&+\alpha_{k}\inner{\hat{V}_{k+1}-\nabla h_{\eta}(y_{k}),\tilde{\scrG}_{\eta,\alpha_{k}}(y_{k})}\\
&\leq h_{\eta}(y_{k})-\alpha_{k}\norm{\tilde{\scrG}_{\eta,\alpha_{k}}(y_{k})}^{2}-(r_{1}(y_{k+1})-r_{1}(y_{k}))+\alpha_{k}\inner{\Delta W_{k+1},\tilde{\scrG}_{\eta,\alpha_{k}}(y_{k})}\\
&+\frac{\alpha^{2}_{k}\lip_{1}(h_{\eta})}{2}\norm{\tilde{\scrG}_{\eta,\alpha_{k}}(y_{k})}^{2}\,,
\end{align*}
where we have used \eqref{eq:ineq_r} in the last inequality. Rearranging the last inequality yields 
\begin{align*}
\Psi_{\eta}(y_{k+1})-\Psi_{\eta}(y_{k})&\leq -\alpha_{k}\norm{\tilde{\scrG}_{\eta,\alpha_{k}}(y_{k})}^{2}\left(1-\frac{\alpha_{k}\lip_{1}(h_{\eta})}{2}\right)+\alpha_{k}\inner{\Delta W_{k+1},\scrG_{\eta,\alpha_{k}}(y_{k})}\\
&+\alpha_{k}\inner{\Delta W_{k+1},\tilde{\scrG}_{\eta,\alpha_{k}}(y_{k})-\scrG_{\eta,\alpha_{k}}(y_{k})}\\
&\leq -\alpha_{k}\norm{\tilde{\scrG}_{\eta,\alpha_{k}}(y_{k})}^{2}\left(1-\frac{\alpha_{k}\lip_{1}(h_{\eta})}{2}\right)+\alpha_{k}\inner{\Delta W_{k+1},\scrG_{\eta,\alpha_{k}}(y_{k})}\\
&+\alpha_{k}\norm{\Delta W_{k+1}}_{\ast}\cdot\norm{\tilde{\scrG}_{\eta,\alpha_{k}}(y_{k})-\scrG_{\eta,\alpha_{k}}(y_{k})}\,,
\end{align*}
where the Cauchy-Schwarz inequality in the last inequality is employed. Using the non-expansiveness of the prox-operator, we obtain 
$$
\norm{\tilde{\scrG}_{\eta,\alpha_{k}}(y_{k})-\scrG_{\eta,\alpha_{k}}(y_{k})}\leq \norm{B^{-1}(\nabla h_{\eta}(y_{k})-V_{k+1})}=\norm{\nabla h_{\eta}(y_{k})-V_{k+1}}_{\ast}=\norm{\Delta W_{k+1}}^{2}_{\ast}.
$$
Hence, we can continue the previous display as
$$
\Psi_{\eta}(y_{k+1})-\Psi_{\eta}(y_{k}) \leq -\alpha_{k}\norm{\tilde{\scrG}_{\eta,\alpha_{k}}(y_{k})}^{2}\left(1-\frac{\alpha_{k}\lip_{1}(h_{\eta})}{2}\right)+\alpha_{k}\inner{\Delta W_{k+1},\scrG_{\eta,\alpha_{k}}(y_{k})}+\alpha_{k}\norm{\Delta W_{k+1}}_{\ast}^{2}.
$$

\subsection{Proof of Lemma \ref{lem:biasestimate}}
For arbitrary $\by\in\scrY$ we compute
\begin{align*}
\Ex_{\Pr}\left[\hat{V}^{\beta}_{\eta,m}(\by)\right]&=\Ex_{\Pr}\left[\frac{1}{m}\sum_{i=1}^{m}\hat{\nabla}_{(U^{i},\eta)}H^{\beta}(\by,\xi^{i})\right]\\
&=\frac{1}{m}\sum_{i=1}^{m}\Ex_{\Pr}\left[\hat{\nabla}_{(U^{i},\eta)}H^{\beta}(\by,\xi^{i})\right]\\
&=\frac{1}{m}\sum_{i=1}^{m}\Ex_{\Pr}\left[\frac{F(\bx^{\beta}(\by+\eta U^{i},\xi^{i}_{2}),\xi_{1}^{i})-F(\bx^{\beta}(\by,\xi^{i}_{2}),\xi_{1}^{i})}{\eta} BU^{i}\right]\\
&=\frac{1}{m}\sum_{i=1}^{m}\Ex_{\Pr}\left[\frac{F( \bx^{\ast}(\by+\eta U^{i}),\xi^{i}_{2}),\xi_{1}^{i})-F(\bx^{\ast}(\by,\xi^{i}_{2}),\xi_{1}^{i})}{\eta} BU^{i}\right]\\
&\quad+\frac{1}{m}\sum_{i=1}^{m}\Ex_{\Pr}\left[\frac{F(\bx^{\beta}(\by+\eta U^{i}),\xi^{i}_{2}),\xi_{1}^{i})-F(\bx^{\ast}(\by+\eta U^{i},\xi^{i}_{2}),\xi_{1}^{i})}{\eta} BU^{i}\right]\\
&\quad-\frac{1}{m}\sum_{i=1}^{m}\Ex_{\Pr}\left[\frac{F(\bx^{\beta}(\by,\xi^{i}_{2}),\xi_{1}^{i})-F(\bx^{\ast}(\by,\xi^{i}_{2}),\xi_{1}^{i})}{\eta} BU^{i}\right].
\end{align*}
From Lemma~\ref{lem:errorbound}, we deduce that
\[ 
\frac{1}{m}\sum_{i=1}^{m}\Ex_{\Pr}\left[\frac{F( \bx^{\ast}(\by+\eta U^{i},\xi^{i}_{2}),\xi^{i}_{1})-F(\bx^{\ast}(\by,\xi^{i}_{2}),\xi^{i}_{1})}{\eta} BU^{i}\right] = \nabla h_{\eta}(\by),
\]
and by mutual independence of $U^{i}$ from $\xi^{i}=(\xi^{i}_{1},\xi^{i}_{2})$ 
\[
\frac{1}{m}\sum_{i=1}^{m}\Ex_{\Pr}\left[\frac{F(\bx^{\beta}(\by,\xi^{i}_{2}),\xi_{1}^{i})-F(\bx^{\ast}(\by,\xi^{i}_{2}),\xi_{1}^{i})}{\eta} BU^{i}\right] = 0. 
\]
For the second assertion, we apply Lipschitz continuity of $F$ (Assumption~\ref{ass:UL}), the iid assumption on the random pair $(U^{i},\xi^{i})$, and Hölder's inequality to obtain
\begin{align*}
 &\frac{1}{m}\sum_{i=1}^m \dnorm{\Ex_{\Pr}\Big[ \frac{F(\bx^{\beta}(\by+\eta U^i,\xi_{2}^{i}),\xi^{i}_{1})-F(\bx^\ast(\by+\eta U^i),\xi^{i}_{1})}{\eta} BU^{i}\Big]}\\ &\le \Ex_{\Pr}\Big[\dnorm{ \frac{F(\bx^{\beta}(\by+\eta U,\xi_{2}),\xi_{1})-F(\bx^\ast(\by+\eta U,\xi_{2}),\xi_{1})}{\eta} BU }\Big] \\
 &=\frac{1}{\eta}\Ex_{\Pr}\Big[\abs{F(x^{\beta}(y+\eta U,\xi_{2}),\xi_{1})-F(x^{\ast}(y+\eta U,\xi_{2}),\xi_{1})}\dnorm{BU}\Big]\\
&\leq \frac{1}{\eta}\Ex_{\Pr}\Big[\lip_{0}(F(\cdot,\xi_{1}))\norm{x^{\beta}(y+\eta U,\xi_{2})-x^{\ast}(y+\eta U,\xi_{2})}_{\scrX}\cdot \dnorm{BU}\Big]\\
&\leq \frac{1}{\eta}\Ex_{\Pr}[\lip_{0}(F(\cdot,\xi_{1}))]\cdot \Ex_{\Pr}\Big[\norm{x^{\beta}(y+\eta U,\xi_{2})-x^{\ast}(y+\eta U,\xi_{2})}_{\scrX}\cdot \dnorm{BU}\Big]\\
&\leq \frac{\abs{\lip_{0}(F(\cdot,\xi_{1})}_{1}}{\eta}\Ex_{\Pr}\Big[\norm{x^{\beta}(y+\eta U,\xi_{2})-x^{\ast}(y+\eta U,\xi_{2})}_{\scrX}^{p}\Big]^{\frac{1}{p}}\cdot\Ex_{\Pr}\Big[\norm{U}^{\frac{p-1}{p}}\Big]^{\frac{p}{p-1}}\\
&\leq \frac{\sqrt{n}\abs{\lip_{0}(F(\cdot,\xi_{1})}_{1}}{\eta}\Ex_{\Pr}\Big[\norm{x^{\beta}(y+\eta U,\xi_{2})-x^{\ast}(y+\eta U,\xi_{2})}_{\scrX}^{p}\Big]^{\frac{1}{p}}.
\end{align*}

\subsection{Proof of Lemma \ref{lem:errorboundinexact}}
\label{proof:lem4.10}
\begin{align*}
\norm{a_{k+1}}_{\ast}&=\frac{1}{m_{k+1}}\norm{\sum_{i=1}^{m_{k+1}}\left[\frac{F(\bx^{\beta_{k}}(y_{k},\xi^{i}_{2,k+1}),\xi_{1,k+1}^{i})-F(\bx^{\ast}(y_{k},\xi^{i}_{2,k+1}),\xi_{1,k+1}^{i})}{\eta}\right]BU^{i}_{k+1}}_{\ast}\\
&\leq \frac{1}{m_{k+1}}\sum_{i=1}^{m_{k+1}}\abs{\frac{F(\bx^{\beta_{k}}(y_{k},\xi^{i}_{2,k+1}),\xi_{1,k+1}^{i})-F(\bx^{\ast}(y_{k},\xi^{i}_{2,k+1}),\xi_{1,k+1}^{i})}{\eta} }\norm{U^{i}_{k+1}}\\
&\leq \frac{1}{m_{k+1}}\sum_{i=1}^{m_{k+1}}\frac{\lip_{0}(F(\cdot,\xi^{i}_{1,k+1}))}{\eta}\norm{\bx^{\beta_{k}}(y_{k},\xi^{i}_{2,k+1})-\bx^{\ast}(y_{k},\xi^{i}_{2,k+1})}_{\scrX}\cdot\norm{U^{i}_{k+1}}
\end{align*}
Hence, by Jensen's inequality and the tower property and the independence of the triple $(\xi^{i}_{1,k+1},\xi^{i}_{2,k+1},U^{i}_{k+1})$, we obtain 
\begin{align*}
\Ex[\norm{a_{k+1}}^{2}_{\ast}\vert\scrF_{k}]&\leq\frac{1}{\eta^{2}m^{2}_{k+1}}\Ex\left[\left(\sum_{i=1}^{m_{k+1}}\lip_{0}(F(\cdot,\xi^{i}_{1,k+1}))\norm{\bx^{\beta_{k}}(y_{k},\xi^{i}_{2,k+1})-\bx^{\ast}(y_{k},\xi^{i}_{2,k+1})}_{\scrX}\cdot\norm{U^{i}_{k+1}}\right)^{2}\vert\scrF_{k}\right]\\
&\leq \frac{1}{\eta^{2}m_{k+1}}\sum_{i=1}^{m_{k+1}}\Ex\left[\lip_{0}(F(\cdot,\xi^{i}_{1,k+1}))^{2}\norm{\bx^{\beta_{k}}(y_{k},\xi^{i}_{2,k+1})-\bx^{\ast}(y_{k},\xi^{i}_{2,k+1})}^{2}_{\scrX}\cdot\norm{U^{i}_{k+1}}^{2}\vert\scrF_{k}\right]\\
&= \frac{n}{\eta^{2}m_{k+1}}\sum_{i=1}^{m_{k+1}}\Ex\left[\lip_{0}(F(\cdot,\xi^{i}_{1,k+1}))^{2}\vert\scrF_{k}\right]\cdot\Ex\left[\norm{\bx^{\beta_{k}}(y_{k},\xi^{i}_{2,k+1})-\bx^{\ast}(y_{k},\xi^{i}_{2,k+1})}^{2}_{\scrX}\vert\scrF_{k}\right]\\
&\leq\frac{n\abs{\lip_{0}(F(\cdot,\xi_{1})}_{2}^{2}}{\eta^{2} m_{k+1}}\sum_{i=1}^{m_{k+1}}\Ex\left[\norm{\bx^{\beta_{k}}(y_{k},\xi^{i}_{2,k+1})-\bx^{\ast}(y_{k},\xi^{i}_{2,k+1})}^{p}_{\scrX}\vert\scrF_{k}\right]^{2/p}\\
&\leq\frac{n\abs{\lip_{0}(F(\cdot,\xi_{1})}_{2}^{2}}{\eta^{2}}\beta^{2}_{k}\eqdef C_{F}\frac{\beta^{2}_{k}}{\eta^{2}},
\end{align*}
where $p\geq 2$, is the exponent from Definition \ref{def:inexact_lower}. 
%
%
We can bound the $L^{2}(\Pr)$-norm for the bias term $b_{k+1}$ in a similar way. First, observe that 
\begin{align*}
\norm{b_{k+1}}_{\ast}&\leq \frac{1}{m_{k+1}}\sum_{i=1}^{m_{k+1}}\abs{\frac{F(\bx^{\beta_{k}}(y_{k}+\eta U^{i}_{k+1},\xi^{i}_{2,k+1}),\xi_{1,k+1}^{i})-F(\bx^{\ast}(y_{k}+\eta U^{i}_{k+1},\xi^{i}_{2,k+1}),\xi_{1,k+1}^{i})}{\eta}}\norm{U^{i}_{k+1}}\\
&\leq \frac{1}{\eta m_{k+1}}\sum_{i=1}^{m_{k+1}}\lip_{0}(F(\cdot,\xi^{i}_{1,k+1}))\norm{\bx^{\beta_{k}}(y_{k}+\eta U^{i}_{k+1},\xi^{i}_{2,k+1})-\bx^{\ast}(y_{k}+\eta U^{i}_{k+1},\xi^{i}_{2,k+1})}_{\scrX}\cdot \norm{U^{i}_{k+1}}.
\end{align*}
Using Jensen's inequality and Hölder's inequality as in the previous estimate, we see for $s\geq 1$, 
\begin{align*}
&\Ex[\norm{b_{k+1}}^{2}_{\ast}\vert\scrF_{k}]\\
&\leq \frac{1}{\eta^{2}m_{k+1}}\sum_{i=1}^{m_{k+1}}\Ex\left[\lip_{0}(F(\cdot,\xi^{i}_{1,k+1}))^{2}\norm{\bx^{\beta_{k}}(y_{k}+\eta U^{i}_{k+1},\xi^{i}_{2,k+1})-\bx^{\ast}(y_{k}+\eta U^{i}_{k+1},\xi^{i}_{2,k+1})}^{2}_{\scrX}\cdot \norm{U^{i}_{k+1}}^{2}\vert\scrF_{k}\right]\\
&= \frac{\abs{\lip_{0}(F(\cdot,\xi^{i}_{1}))}^{2}_{2}}{\eta^{2}m_{k+1}}\sum_{i=1}^{m_{k+1}}\Ex\left[\norm{\bx^{\beta_{k}}(y_{k}+\eta U^{i}_{k+1},\xi^{i}_{2,k+1})-\bx^{\ast}(y_{k}+\eta U^{i}_{k+1},\xi^{i}_{2,k+1})}^{2}_{\scrX}\cdot\norm{U^{i}_{k+1}}^{2}\vert\scrF_{k}\right] \\
& \leq  \frac{\abs{\lip_{0}(F(\cdot,\xi^{i}_{1}))}^{2}_{2}}{\eta^{2}m_{k+1}}\sum_{i=1}^{m_{k+1}}\Ex\left[\norm{\bx^{\beta_{k}}(y_{k}+\eta U^{i}_{k+1},\xi^{i}_{2,k+1})-\bx^{\ast}(y_{k}+\eta U^{i}_{k+1},\xi^{i}_{2,k+1})}^{2s}_{\scrX}\vert\scrF_{k}\right]^{1/s}\cdot\Ex[\norm{U^{i}_{k+1}}^{2r}]^{1/r}
\end{align*}
for $\frac{1}{s}+\frac{1}{r}=1$. Choosing $2s=p$, we obtain 
\begin{align*}
&\Ex[\norm{b_{k+1}}^{2}_{\ast}\vert\scrF_{k}]\\
& \leq  \frac{\abs{\lip_{0}(F(\cdot,\xi^{i}_{1}))}^{2}_{2}}{\eta^{2}m_{k+1}}\sum_{i=1}^{m_{k+1}}\Ex\left[\norm{\bx^{\beta_{k}}(y_{k}+\eta U^{i}_{k+1},\xi^{i}_{2,k+1})-\bx^{\ast}(y_{k}+\eta U^{i}_{k+1},\xi^{i}_{2,k+1})}^{p}_{\scrX}\vert\scrF_{k}\right]^{2/p}\cdot\Ex[\norm{U^{i}_{k+1}}^{\frac{2p}{p-2}}]^{\frac{p-2}{p}}\\
& \leq  \frac{n\abs{\lip_{0}(F(\cdot,\xi^{i}_{1}))}^{2}_{2}}{\eta^{2}}\beta^{2}_{k}=C_{F}\frac{\beta^{2}_{k}}{\eta^{2}}.
\end{align*}

 \section{Monotonicity of the prox-gradient mapping}
 \label{app:2}
%

Consider the function $\varphi_{\by}:\alpha\mapsto \frac{1}{\alpha}\norm{y-T_{\eta,\alpha}(\by)}$. For $\by\in\zer(\partial r_{1}+\nabla h_{\eta})$, if we have $\varphi_{\by}(\alpha)=0$ for all $\alpha>0$. We next prove a classical monotonicity result with respect to the parameter $\alpha$ of this mapping. 
\begin{proposition}
If $\by\notin\zer(\partial r_{1}+\nabla h_{\eta})$, then 
\begin{equation}
\alpha_{1}>\alpha_{2}>0\Rightarrow \varphi_{\by}(\alpha_{1})<\varphi_{\by}(\alpha_{2}). 
\end{equation}
\end{proposition}

\begin{proof}
To simplify notation, let us define $\bar{\by}(\alpha):=T_{\eta,\alpha}(\by)$. This point satisfies the monotone inclusion (Fermat's optimality principle)
\begin{align*}
\frac{1}{\alpha}B(y-\bar{y}(\alpha))-\nabla h_{\eta}(\by)\in\partial r_{1}(\bar{\by}(\alpha)). 
\end{align*}
Hence, for $\alpha_{1}>\alpha_{2}>0$, the maximal monotonicity of the subdifferential $\partial r_{1}$ yields
\[
\inner{\frac{1}{\alpha_{1}}B(\by-\bar{\by}(\alpha_{1}))-\frac{1}{\alpha_{2}}B(\by-\bar{\by}(\alpha_{2})),\bar{\by}(\alpha_{1})-\bar{\by}(\alpha_{2})}\geq 0. 
\]
Rearranging, 
\begin{align*}
0&\leq \frac{1}{\alpha_{1}}\inner{B(\by-\bar{\by}(\alpha_{1})),\bar{\by}(\alpha_{1})-\by}+\frac{1}{\alpha_{1}}\inner{B(\by-\bar{\by}(\alpha_{1})),\by-\bar{\by}(\alpha_{2})}\\
&-\frac{1}{\alpha_{2}}\inner{B(\by-\bar{\by}(\alpha_{2})),\by-\bar{\by}(\alpha_{1})}\\
&-\frac{1}{\alpha_{2}}\inner{B(\by-\bar{\by}(\alpha_{2})),\by-\bar{\by}(\alpha_{2})}\\
&=-\frac{1}{\alpha_{1}}\norm{\bar{\by}(\alpha_{1})-\by}^{2}-\frac{1}{\alpha_{2}}\norm{\bar{\by}(\alpha_{2})-\by}^{2}\\
&+\left(\frac{1}{\alpha_{1}}+\frac{1}{\alpha_{2}}\right)\inner{B(\by-\bar{\by}(\alpha_{1})),\by-\bar{\by}(\alpha_{2})}
\end{align*}
Consequently, 
\begin{align*}
\alpha_{1}\varphi_{\by}(\alpha_{1})^{2}+\alpha_{2}\varphi_{\by}(\alpha_{2})^{2}&\leq (\alpha_{1}+\alpha_{2})\inner{B\left(\frac{\by-\bar{\by}(\alpha_{1})}{\alpha_{1}}\right),\frac{\by-\bar{\by}(\alpha_{2})}{\alpha_{2}}}\\
&\leq \frac{\alpha_{1}+\alpha_{2}}{2}\left(\varphi_{\by}(\alpha_{1})^{2}+\varphi_{\by}(\alpha_{2})^{2}\right).
\end{align*}
Rearranging, we see that 
$$
(\alpha_{1}-\alpha_{2})\left(\varphi_{\by}(\alpha_{1})^{2}-\varphi_{\by}(\alpha_{2})^{2}\right)\leq 0
$$
\end{proof}

 \end{appendix}
 
\bibliographystyle{plain}
\bibliography{mybib}

 \end{document}